\documentclass[11pt, a4paper]{amsart}

\usepackage{xcolor}
\definecolor{MyBlue}{cmyk}{1,0.13,0,0.63}
\definecolor{MyGreen}{cmyk}{0.91,0,0.88,0.52}
\newcommand{\mylinkcolor}{MyBlue}
\newcommand{\mycitecolor}{MyGreen}
\newcommand{\myurlcolor}{black}

\usepackage{hyperref}
\hypersetup{%
  bookmarksnumbered=true,bookmarksopen=false,%
  plainpages=false,
  linktocpage=true,%
  colorlinks=true,breaklinks=true,%
  linkcolor=\mylinkcolor,citecolor=\mycitecolor,urlcolor=\myurlcolor,%
  pdfpagelayout=OneColumn,%
  pageanchor=true,%
}

\usepackage{graphicx}
\usepackage{amsmath, amsthm,  amsfonts, amscd}
\usepackage{amssymb}
\usepackage{url}
\usepackage{fullpage}
\usepackage{mathtools}
\usepackage[all]{xy}
\usepackage{slashed}
\usepackage{multirow}

\newtheorem{thm}{Theorem}[section]
\newtheorem*{thm*}{Theorem}
\newtheorem{cor}[thm]{Corollary}
\newtheorem{lemma}[thm]{Lemma}
\newtheorem{prop}[thm]{Proposition}

\theoremstyle{definition}
\newtheorem{defn}[thm]{Definition}

\theoremstyle{remark}
\newtheorem{remark}[thm]{Remark}
\newtheorem*{notation}{Notation}
\newtheorem{example}[thm]{Example}
\newtheorem{remarks}[thm]{Remarks}
\newtheorem{examples}[thm]{Examples}

\newcommand{\End}{\ensuremath{\mathrm{End}}}

\newcommand{\wh}{\ensuremath{\widehat}}

\newcommand{\R}{\ensuremath{\mathbb{R}}}
\newcommand{\N}{\ensuremath{\mathbb{N}}}
\newcommand{\Z}{\ensuremath{\mathbb{Z}}}

\newcommand{\C}{\ensuremath{\mathbb{C}}}
\newcommand{\T}{\ensuremath{\mathbb{T}}}

\def\calT{\mathcal{T}}
\def\calC{\mathcal{C}}
\def\calL{\mathcal{L}}

\def\calB{\mathcal{B}}
\def\calH{\mathcal{H}}

\def\calF{\mathcal{F}}
\def\calA{\mathcal{A}}

\def\calS{\mathcal{S}}
\def\calE{\mathcal{E}}

\def\calW{\mathcal{W}}
\def\calU{\mathcal{U}}

\def\calG{\mathcal{G}}

\def\bP{\mathbf{P}}

\newcommand{\frakh}{{\mathfrak h}}

\renewcommand{\epsilon}{\varepsilon}

\newcommand{\ol}{\overline}

\theoremstyle{definition}

\DeclareMathOperator{\Dom}{Dom}

\DeclareMathOperator{\Tr}{Tr}

\newcommand{\A}{\mathcal{A}}

\newcommand{\rst}[1]{\ensuremath{{\mathbin\upharpoonright}%
\raise-.5ex\hbox{$#1$}}}

\makeatletter

\newcommand{\Rmnum}[1]{\expandafter\@sl217--242owromancap\romannumeral #1@}
\makeatother

\author{C. Bourne}
\address{WPI-AIMR, Tohoku University,
2-1-1 Katahira, Aoba-ku, Sendai, 980-8577, Japan \emph{and} 
RIKEN iTHEMS, 2-1 Hirosawa, Wako, Saitama 351-0198, Japan}
\email{chris.bourne@tohoku.ac.jp}

\author{B. Mesland}
\address{Mathematisch Instituut, Universiteit Leiden,
Niels Bohrweg 1, 2333 CA Leiden, The Netherlands}
\email{b.mesland@math.leidenuniv.nl}

\thanks{Keywords: Operator algebras, Groupoid and Morita equivalence, Gabor analysis, Wannier basis}
\thanks{MSC (2010): 46L08, 46L55, 81R60}

\date{\today}

\begin{document}

\title{Localised module frames and Wannier bases from groupoid Morita equivalences}

\begin{abstract}
Following the operator algebraic approach to Gabor analysis, 
we construct frames of translates for the Hilbert space localisation of 
the Morita equivalence bimodule arising from a groupoid equivalence between Hausdorff groupoids, where one 
of the groupoids is \'{e}tale and with a compact unit space.
For finitely generated and projective submodules, we show these  frames are orthonormal 
bases if and only if the module is free. We then apply 
this result to the study of localised Wannier bases of spectral subspaces of 
Schr\"{o}dinger operators with atomic potentials supported on (aperiodic) Delone sets. 
The noncommutative Chern numbers provide a topological obstruction to fast-decaying Wannier bases 
and we show this result is stable under deformations of the underlying Delone set. 
\end{abstract}

\maketitle

\parindent=0.0in
\parskip=0.04in

\vspace{-0.2cm}

\section{Introduction}

A key question in time-frequency analysis and related fields is the reconstruction of elements in a Hilbert space $\frakh$
via a set of vectors 
$\{T_j w_1,\ldots,T_j w_m\}_{j\in J}$ spanning $\frakh$ and
where $\{T_j\}_{j\in J} \subset \calB(\frakh)$ is some canonically defined 
set of operators. An important example  are (multi-window) Gabor frames, 
where given a locally compact and abelian group $G$, a frame on 
$L^2(G)$ is constructed via translation and modulation operators from 
a closed subset $\Lambda \subset G\times \wh{G}$ on a window function.
Morita equivalence bimodules of $C^*$-algebras, also called 
imprimitivity bimodules, have been shown to be a useful tool in the 
construction of such Gabor frames~\cite{Luef09, AE20, AJL19}. 
One may also consider frames generated from translations by elements of a 
discrete group or \'{e}tale groupoid, which we call \emph{frames of translates}, cf.~\cite[Chapter 9]{Christensen} or~\cite{BMS16}.

Closely related to frames of translates and arising from physics are Wannier bases. 
Given a Schr\"{o}dinger-type operator $H$ with periodic potential acting 
on $L^2(\R^d, \C^n)$,
a Wannier basis is an orthonormal 
basis of a spectral subspace of $H$ 
constructed from a finite set of functions along with their translations in $\Z^d$. 
Because the operator $H$ has a periodic potential, such bases exist by the Bloch--Floquet transform. 
The regularity of Wannier bases change drastically depending on the topological properties 
of the spectral band of the Schr\"{o}dinger operator, where a delocalised Wannier basis can be 
used as an indicator that the system has a non-trival topological phase, 
see~\cite{MagneticWannier, DeNittisLeinWannier, Kuchment, HaldaneWannier, WannierDecay} for example. 
Wannier bases with exponential decay can be constructed for periodic and 
aperiodic Hamiltonians such that the compression of a position operator by the Fermi 
projection has uniform spectral gaps~\cite{SWL1, SWL2}.

In the context of periodic systems with a space group translation symmetry $G \subset \R^d$, 
it was shown by Ludewig and Thiang that the existence of a fast-decaying Wannier basis is equivalent to whether a 
finitely generated and projective $C^*_r(G)$-module is free or not~\cite{LudewigThiang}.

The purpose of this paper is twofold. 
First, using a similar framework to~\cite{AJL19}, we study the relation between 
frames of translates and 
Morita equivalence bimodules  arising from groupoid equivalences. 
We then use this relation to extend the results of Ludewig and Thiang~\cite{LudewigThiang} 
on fast-decaying Wannier bases  to  \'{e}tale groupoids. Regularity of frames  
is examined using
pre-Morita equivalence bimodules of algebras defined from derivations and differential seminorms.

For both of our aims, our guiding example is the \'{e}tale groupoid $\calG_\mathrm{Del}$ constructed from 
a Delone set $\Lambda \subset \R^d$ and which is equivalent to the crossed product groupoid 
of the translation action on the orbit space of $\Lambda$,  $\Omega_\Lambda \rtimes \R^d$. 
In previous work~\cite{ BMes, BP17}, the index theory of $\calG_\mathrm{Del}$ was studied and 
its application to aperiodic topological phases. 
In~\cite{Kreisel},  Gabor frames of $L^2(\R^d)$ were constructed from 
Delone subsets of $\R^{2d}$ with finite local complexity and the groupoid $\calG_\mathrm{Del}$
using results from~\cite{GOR15}. Gabor duality was  shown for Gabor frames constructed 
from model sets in~\cite{ModelSetGabor}.
We  do not consider the important question of 
Gabor duality in this work.

\subsection*{Outline and main results}

Because our results bring together constructions from time-frequency analysis, groupoids, 
$C^*$-modules and Morita equivalence, we give a brief overview of these concepts in Section \ref{sec:Prelims}. 
In Section \ref{subsec:derivations_and_bimodules}, 
using the framework of differential seminorms (cf.~\cite{BC91}), we construct pre-Morita equivalence 
bimodules for pre-$C^\ast$-subalgebras obtained from a finite family of commuting unbounded $*$-derivations.

In Section \ref{sec:Gpoid_Gabor_General} we consider a groupoid equivalence $\calG \leftarrow Z \to \calH$ 
of Hausdorff, second countable and locally 
compact groupoids $\calG$ and $\calH$, where $\calH$ is 
\'{e}tale with a compact unit space $\calH^{(0)}$. By choosing the evaluation state for some $x \in \calH^{(0)}$ we obtain 
a translation action of the fibre $r^{-1}(x)$ on the Hilbert space localisation $\frakh_x$ of the 
Morita equivalence bimodule between $C^*_r(\calG)$ and $C^*_r(\calH)$. 
This allows us to construct a normalised tight frame of translates for $\frakh_x$ from the 
$C^*$-module frame of the Morita equivalence bimodule. Restricting to finitely 
generated and projective $C^{*}_{r}(\calH)$-submodules, 
we obtain a frame of translates with a finite generating set 
for a subspace of the Hilbert space localisation $\frakh_x$. This frame is an 
orthonormal basis for all $x\in \calH^{(0)}$ if and only if the finitely generated and 
projective module is free, and thus its class in the reduced $K$-theory of $C^{*}_{r}(\calH)$ is trivial. 
These results are extended in Section \ref{sec:Twsited_Gabor} 
to the case of abstract transversals with a normalised $2$-cocycle twist.

We apply these results to the \'{e}tale Delone groupoid in  Section \ref{Sec:Delone_application}. 
We consider a magnetic Schr\"{o}dinger operator with an atomic potential $v$ arranged on a 
Delone set $\Lambda \subset \R^d$,
$$
  H_\Lambda =  \sum_{j=1}^d \!\Big(-i\frac{\partial}{\partial x_j} - A_j \Big)^2 
  + \sum_{p \in \Lambda} v(\cdot - p),
$$
where $A$ is the magnetic potential. Results by Bellissard et al. show that for 
sufficiently regular  $v$, $H_\Lambda$ 
and its magnetic translates are affiliated to the 
crossed product $C^*$-algebra $C^*_r(\Omega_\Lambda \rtimes \R^d, \theta)$ with 
$\theta$ a magnetic twist~\cite{Bel92,BHZ00}. We show that for any Delone set $\calL$ in the transversal 
subset $\Omega_0 \subset \Omega_\Lambda$, 
a gapped spectral subspace of $H_\calL$ 
has a normalised tight  frame built from the magnetic translations 
$\{U_y\}_{y\in \calL}$. This frame is an orthonormal basis if and only if 
the corresponding finitely generated and projective $C^*_r(\calG_\mathrm{Del}, \theta)$-module 
is free. Using derivations of the groupoid algebras and differential seminorms, 
we show this normalised 
tight frame  has faster than polynomial decay. We therefore see that 
topological properties of spectral subspaces of the Delone Schr\"{o}dinger operator can 
be related to the regularity of (aperiodic) Wannier bases.

The regularity of such Wannier basis is closely related to the Localisation Dichotomy Conjecture for non-periodic 
insulators raised in~\cite[Section 5 (arXiv version)]{HaldaneWannier} and further 
studied in~\cite{PanatiDeloneWannier, LuStubbs21}. We 
prove  a weaker version  of this conjecture in Section \ref{subsec:Loc_dichotomy} 
and show that the existence of Wannier bases with faster than polynomial decay is 
equivalent to the existence of Wannier bases such that $\sum_j (1+|x|^2) |w_j (x)|^2 \in L^1(\R^d)$. This in turn 
is equivalent to the spectral projection defining a freely generated $C^*_r(\calG_\mathrm{Del}, \theta)$-module.
Our results are weaker than those posed in~\cite{HaldaneWannier} as we consider a family of 
Schr\"{o}dinger operators and Hilbert spaces rather than a single Hamiltonian. Similarly we 
do not consider Wannier bases with exponential decay.

By relating the existence of a localised Wannier basis to a $K$-theoretic property,  
the noncommutative Chern numbers for $C^*_r(\Omega_\Lambda \rtimes \R^d, \theta)$ 
and $C^*_r(\calG_\mathrm{Del}, \theta)$ studied in~\cite{BMes, BP17, BRCont} give a 
topological  obstruction to a Wannier basis with fast decay. 
We also show that these Chern number formulas are continuous under deformations of the 
magnetic field or Delone atomic potential provided the spectral gap stays open. This implies that
 a non-localised Wannier basis is stable under deformations of the 
 atomic potential (e.g. from a periodic lattice to an 
aperiodic point pattern).

\section{Preliminaries} \label{sec:Prelims}

\subsection{Hilbert space frames}
Let us recall a few basic definitions from time-frequency analysis.

\begin{defn}
Let $\mathfrak{h}$ be a Hilbert space and $\{g_j\}_{j\in J}$ is a collection of elements in $\mathfrak{h}$. 
We say that $\{g_j\}_{j\in J}$ form a Hilbert space frame if there are 
constants $C,\, D>0$ such that 
$$
   C \|\psi \|^2 \leq \sum_{j\in J} \big| \langle g_j, \psi \rangle \big|^2 \leq D \|\psi\|^2 \quad \text{for all }\psi \in \mathfrak{h}.
$$
If $C=D$, then $\{g_j\}_{j \in J}$ is called a tight frame. 
If $C=D=1$, then $\{g_j\}_{j \in J}$ is called a normalised tight frame or Parseval frame.
\end{defn}

Orthonormal bases are obvious examples of  normalised tight frames. The restriction of an 
orthonormal basis of a Hilbert space to a closed subspace yields a normalised  tight frame for the 
subspace. Normalised tight frames always arise as a compression of an 
orthonormal basis.

\begin{prop}[\cite{HanLarson}, Section 1]
Let $\{g_j \}_{j\in J}$ be a normalised tight frame of a Hilbert space $\frakh_1$. Then there is
Hilbert space $\frakh_2$ and an orthonormal basis $\{e_j\}_{j\in J}$ of $\frakh_1 \oplus \frakh_2$ 
such that  $g_j = \mathrm{pr}_1(e_j)$.
\end{prop}

Any Hilbert space frame yields an invertible 
frame operator 
$$
 S:\mathfrak{h} \to \mathfrak{h}, \qquad 
 S(\psi) = \sum_{j\in J} g_j \langle  g_j , \psi  \rangle .
$$
We note that, in contrast to the Gabor analysis literature, we work with inner-products 
that are linear on the right. This is to make our results more easily compatible with 
right $C^*$-modules and their Hilbert space localisations (Definition \ref{def:HS_localisation} below). 
Because $S$ is invertible, one obtains a reconstruction formula for elements in $\mathfrak{h}$.
$$
   \psi = \sum_{j\in J} S^{-1} g_j \langle g_j, \psi \rangle, \quad \psi \in\mathfrak{h}.
$$
The elements $\{S^{-1}g_j\}_{j\in J}$ are called the dual frame to $\{g_j\}_{j\in J}$.

\subsection{Pre-$C^*$-modules and Morita equivalence} \label{subsec:C_mod_MEBs}

Following a similar framework to~\cite{AJL19}, we will use $C^*$-modules and 
Morita equivalence bimodules 
to study questions around frames of translates for their Hilbert space localisations. We now recall some basic definitions and establish 
notation. Further details can be found in~\cite{Blackadar, RaeburnWilliams}.

We will also be interested in the case where the $C^*$-algebras $A$ and 
$B$ have dense $\ast$-subalgebras $\calA$ and $\calB$.

\begin{defn} Let $B$ be a $C^*$-algebra, $\calB\subset B$ a dense $*$-subalgebra and 
$\calE_{\calB}$ a vector space that is also a right $\calB$-module. We say that $\mathcal{E}_{\mathcal{B}}$ is 
a right  inner product $\mathcal{B}$-module if there is a  sesquilinear pairing 
$(e_1,e_{2})\mapsto (e_1\mid e_2)_{\mathcal{B}}\in \mathcal{B}$  linear in the second variable
such that for $e_1,e_2,e\in \mathcal{E}$
\begin{align*}
  &(e_1 \mid e_2)_\mathcal{B} = (e_2 \mid e_1)_\mathcal{B}^*, 
  &&(e_1 \mid e_2\cdot b)_\mathcal{B} = (e_1\mid e_2)_\mathcal{B} \, b,  \\
  &(e\mid e)_\mathcal{B} \geq 0 \in B,
  &&(e\mid e)_{\mathcal{B}}= 0\Rightarrow e=0.
\end{align*}
An inner product module $\mathcal{E}_{\mathcal{B}}$ is called full if $\mathrm{span}\{ (e_1\mid e_2)\,:\, e_1, e_2\in \mathcal{E}_{\mathcal{B}}\} $ is $C^{*}$-norm dense in $B$. 
If $B$ is a $C^{*}$-algebra, a $C^{*}$-module is a right inner product $B$-module that is complete in the norm $\|e\|^{2}:=\|(e\mid e)_{B}\|_{B}$. 
\end{defn}

 Given a $C^*$-algebra $A$ and dense $\ast$-subagebra $\calA$, a left inner product $\calA$-module ${}_{\calA}\calE$ can be 
analogously defined, where the sesquilinear pairing $(e_1,e_2)\mapsto {}_{\calA}(e_1 \mid e_2)$ is linear in the 
first variable and ${}_{\calA}(a\cdot e_1 \mid e_2) = a \,{}_{\calA}(e_1\mid e_2)$. 
 In case the algebras $\calA$ and $\calB$ and left/right inner product module structure are clear from context, 
 we will suppress subscripts and write $\calE$ for $_\calA\calE_\calB$.

On an inner product module $\mathcal{E}_{\mathcal{B}}$, the norm $\|e\|^{2}:=\|(e\mid e)_{B}\|_{B}$ is defined and the completion of $\mathcal{E}_{\mathcal{B}}$ in this norm is a right $C^{*}$-module $E_{B}$ over $B$. For a right inner product $\mathcal{B}$-module, the $*$-algebra of finite rank operators $\textnormal{Fin}_{\mathcal{B}}(\mathcal{E}_{\mathcal{B}})$ is defined
to be the algebraic span of the finite-rank operators $\{\Theta^R_{e_1,e_2}\}_{e_1,e_2\in \mathcal{E}},$ where
$$
  \Theta^R_{e_1,e_2}(e_3) = e_1 \cdot (e_2\mid e)_\calB, \quad \left(\Theta^R_{e_1,e_2}\right)^{*}:=\Theta^R_{e_2,e_1},\quad e_1,\, e_2,\, e \in \mathcal{E}.
$$
When $E_{B}$ is a $C^{*}$-module over a ${*}$-algebra $B$,
the compact endomorphisms $\mathbb{K}_B(E)$ are defined as the operator norm closure of $\textnormal{Fin}_{B}(E)$. 
It is a closed two-sided ideal in the $C^{*}$-algebra $\End^*_B(E)$ of adjointable operators 
on $E_B$. 

\begin{defn}
Let $\calE_{\calB}$ be a right inner product $\calB$-module. A set $\{e_{1},\cdots, e_n\}$ is called a 
{finite} module frame if 
\[\mathrm{Id}_{\calE}=\sum_{k=1}^{n}\Theta^{R}_{e_{k},e_{k}}.\]
If $E_B$ is a right $C^*$-module, a countable subset $\{e_j\}_{j\in J} \subset E_B$ is a (right) $C^*$-module 
frame if $\sum_j \Theta^R_{e_j,e_j}$ converges strictly to $\mathrm{Id}_E$.
\end{defn}

We remark that any countably generated right $C^*$-module over a $\sigma$-unital 
algebra $B$ admits a $C^*$-module frame~\cite{Blackadar, RaeburnWilliams}. 
If an inner product module $\calE_\calB$ admits a finite frame $\{v_j\}_{j=1}^n$, then there is a 
projection $p =p^* =p^2 \in M_n(\calB)$ and right module maps
\begin{align}
\label{eq:finite_frame_fgp_equiv}
   &S: \calE \to \calB^n, &&R: \calB^n \to \calE, \\
   &S(e) = \big( ( v_j\mid e)_\calB \big)_{j=1}^n, 
   &&R \big( (b_j)_{j=1}^n \big) = \sum_{j=1}^n v_j\cdot b_j,
\end{align}
that restrict to isomorphisms $S:\calE\to p\calB^n$ and $R:p\calB^{n}\to \calE$. In particular 
$$
p=(v_i\mid v_j)_{\calB}\in M_{n}(\calB),\quad R\circ S=\mathrm{Id}_{\calE},\quad e= R\circ S(e) = \sum_{j=1}^n v_j\cdot (v_j\mid e)_{\calB} = \sum_{j=1}^n \Theta_{v_j,v_j}^R(e).
$$
Similar formulas hold for left modules with a finite frame.

\begin{defn}
Let $A$ and $B$ be $C^{*}$-algebras. An $A$-$B$ Morita equivalence bimodule is a 
full right-$B$ $C^*$-module and full left-$A$ $C^*$-module 
${}_A E_B$ such that 
\begin{align*}
  &(a \cdot e_1 \mid e_2)_B = (e_1 \mid a^*\cdot e_2)_B, 
  &&{}_A( e_1 \mid e_2 \cdot b) = {}_A(e_1 \cdot b^* \mid e_2), 
  &&{}_A(e_1\mid e_2)\cdot e_3 = e_1 \cdot (e_2 \mid e_3)_B
\end{align*}
for all $a\in A$, $b\in B$ and $e_1,e_2,e_3\in E$. We say that $A$ and $B$ are Morita equivalent 
if there is a Morita equivalence bimodule ${}_A E_B$.
\end{defn}

To distinguish left and right inner products, for $e_1,\,e_2\in E$ we use the notation
\begin{align*}
  &\Theta_{e_1,e_2}^L(e_3) = {}_{A}(e_3\mid e_1)\cdot e_2, 
  &&\Theta_{e_1,e_2}^R(e_3) = e_1 \cdot (e_2\mid e_3)_B.
\end{align*}
A full right-$B$ $C^*$-module is a Morita equivalence bimodule between $\mathbb{K}_B(E)$ and 
$B$ with the $\mathbb{K}_B(E)$-valued inner product 
${}_{\mathbb{K}(E)}(e_1\mid e_2) = \Theta^R_{e_1,e_2}$. Hence 
$A$ is Morita equivalent to $B$ if and only if there is a full right-$B$ $C^*$-module 
$E_B$ and a $\ast$-isomorphism $\phi: A \to \mathbb{K}_B(E)$.

\begin{defn} \label{def:pre_MEB}
Let $\calA$ and $\calB$ be dense $*$-subalgebras of $C^{*}$-algebras $A$ and $B$. A 
pre-Morita equivalence bimodule is an $\calA$-$\calB$ bimodule $_\calA\calE_\calB$, equipped with a 
full left-$\calA$ valued and a full right-$\calB$ valued inner product such that for any $a\in\calA$, $b\in \calB$ 
and $e,e_1,e_2,e_3 \in \calE$ the compatibility conditions
\begin{align*}
  &{}_{\calA}( e\cdot b\mid e\cdot b) \leq \|b\|^2 {}_{\calA}(e\mid e), 
  &&(a \cdot e \mid a\cdot e)_\calB \leq \|a\|^2 (e\mid e)_\calB, 
  &&{}_{\calA}( e_1 \mid e_2) \cdot e_3 = e_1 \cdot (e_2 \mid e_3)_\calB,
\end{align*}
are satisfied. Here $\|\cdot\|$ denotes the $C^{*}$-norm on the algebras $\calA$ and $\calB$, and the inequalities are in the $C^{*}$-algebras $A$ and $B$. 
\end{defn}

As expected, a pre-Morita equivalence 
bimodule ${}_{\calA}\calE_\calB$ can be completed to a Morita equivalence bimodule 
${}_A E_B$, using either of the norms $\|e\|^{2}:=\|(e\mid e)_{\calB}\|_{B}$ or $\|e\|^{2}:=\|{}_{\calA} (e\mid e)\|_{A}$ see~\cite[Section 3.1]{RaeburnWilliams}. 
We proceed with some definitions and results concerning frames in dense submodules of $C^{*}$-modules. For this we need our dense $*$-subalgebras to be equipped with additional analytic structure.

We provide a definition of smooth subalgebra of a $C^{*}$-algebra (see for instance \cite[Definitions 3.25 and 3.26]{ElementsNCG}) that is flexible enough for our purposes.
\begin{defn}
\label{def:smooth_subalgebra}
We say that a $\ast$-algebra $\calA$ is a pre-$C^*$-algebra if it is
\begin{enumerate}
\item[(i)] Fr\'echet, i.e. complete and metrizable such that the multiplication is jointly continuous;
\item[(ii)] Isomorphic to a proper dense $\ast$-subalgebra $\iota(\calA)$ of a $C^*$-algebra $A$, 
where $\iota:\calA\hookrightarrow A$ is the inclusion map, and $\iota(\calA)$ is stable under the 
holomorphic functional calculus. That is, if $f$ is a holomorphic function on a neighbourhood 
of the spectrum of $a\in\iota(\calA)$, then $f(a)\in\iota(\calA)$.
\end{enumerate}
\end{defn}

Stability under the holomorphic functional calculus extends to nonunital algebras, 
since the spectrum of an element in a nonunital algebra is defined to be the spectrum 
of this element in the one-point unitization, though we must restrict to functions 
satisfying $f(0) = 0$. Similarly, the definition of a Fr\'{e}chet algebra does not require a unit.

The $K$-theory groups $K_0(\calA)$ and $K_1(\calA)$ can be defined for 
a pre-$C^*$-algebra $\calA$, see~\cite[Section 3.8]{ElementsNCG} or~\cite{Blackadar} for example. A useful feature of pre-$C^*$-algebras is 
that they contain all the $K$-theoretic information of their $C^*$-completion.

\begin{prop}[\cite{Schweitzer92}] \label{prop:smooth_algebras_give_k_theory_iso}
If $\calA$ is a pre-$C^*$-algebra with $C^*$-completion $A$, then the map induced by the 
inclusion $\iota_\ast: K_j(\calA) \to K_j(A)$ is an isomorphism for $j=0,1$.
\end{prop}

\begin{lemma}
\label{lem: finite_frame}
 Let $\calA$ and $\calB$ be pre-$C^{*}$-algebras with $\calB$ unital and ${}_{\calA}\calE_\calB$ a pre-Morita equivalence bimodule. 
 \begin{enumerate}
 \item There is a finite left module 
frame $\{g_1,\ldots, g_n\} \subset  \calE$ and $1_{\calB}=\sum_{k=1}^{n} (g_{k}\mid g_{k})_{\calB}$.
\item For any $p=p^{*}=p^{2}\in M_{n}(\mathcal{A})$, $p\calE^{\oplus n}$ is a 
finitely generated and projective $\calB$-module and there exists a finite right module frame 
$\{v_{1},\cdots, v_{m}\}\subset p\calE^{\oplus n}$. 
\end{enumerate}
\end{lemma}
\begin{proof}
We write $A$ and $B$ for the $C^{*}$-algebra closures of $\calA$ and $\calB$ and $E$ for the $C^{*}$-module closure of $\mathcal{E}$, which is a Morita equivalence bimodule for the $C^{*}$-algebras $A$ and $B$. 

(1) The $*$-algebra of  finite rank operators $\textnormal{span}\{\Theta^L_{e_1,e_2}\}_{e_1,e_2\in \calE}$ 
is $C^{*}$-norm dense in ${}_{A}\mathbb{K}(E)$. Thus for $\varepsilon < 1$ 
we can find $\{e_1,\cdots , e_{n}\}\subset \calE$ such that the operator 
\[g:=\sum_{k=1}^{n} \Theta^{L}_{e_{k},e_{k}}=\sum_{k=1}^{n} (e_{k}\mid e_{k})_{\calB}\in \calB,\] 
satisfies $\|1_{B}-g\|_{B}<\varepsilon$. Hence the positive element $g$ is invertible in $B$. 
Because the spectrum of $g$ is contained in the region of analycity of $f(z) = z^{-1}$ 
and $\calB$ is stable under the holomorphic functional calculus, $g^{-1}$ and $g^{-1/2} \in \calB$. 
Define $g_{k}:=e_{k}\cdot g^{-\frac{1}{2}}\in\calE$, for $1,\cdots, n$ so that
 \[\sum_{k=1}^{n}(g_{k}\mid g_{k})_{\calB}=\sum_{k=1}^{n} \Theta^{L}_{g_{k},g_{k}}=\sum_{k=1}^{n} \Theta^{L}_{e_{k}g^{-\frac{1}{2}},e_{k}g^{-\frac{1}{2}}}=g^{-\frac{1}{2}}\left(\sum_{k=1}^{n} \Theta^{L}_{e_{k},e_{k}} \right)g^{-\frac{1}{2}}=g^{-\frac{1}{2}}gg^{-\frac{1}{2}}=1_{\calB},\]
 which proves the claim.

 (2) Since $p$ is a compact operator on $E^{\oplus n}$, it is finite rank 
 by~\cite[Corollary 3.10]{ElementsNCG}, so the module $W:=pE^{\oplus n}$ is 
 finitely generated and projective over $B$. Write $\mathcal{W}:=p\calE^{\oplus n}$ 
 for the dense $\calB$-submodule defined by $p\in M_{n}(\calA)$. 
 
The $*$-algebra of  finite rank operators 
$\textnormal{span}\{\Theta^R_{e_1,e_2}\}_{e_1,e_2\in \calW}$ is $C^{*}$-norm dense 
in $\mathbb{K}_{B}(W)$.  Hence for $\varepsilon< 1$ there exist 
$w_{1},\cdots, w_{m}\in \calW$ such that $w:=\sum \Theta^{R}_{w_{k},w_{k}}$ 
satisfies $\|p-w\|_{\mathbb{K}_{B}(W)}<\varepsilon$. It follows that $w$ is 
invertible in the unital $C^{*}$-algebra $\mathbb{K}_{B}(W)\cong p M_{n}(A)p$ 
and has spectrum contained in $B(1;\varepsilon)$. The spectrum of $w$ in $M_n(A)$ 
is thus contained in the disconnected set $B(1;\varepsilon)\cup \{0\}$. 
By spectral invariance, the same holds for the spectrum of $w\in M_{n}(\mathcal{A})$.  
Thus there is a function $f$, holomorphic on a neighborhood of the spectrum of $w$ 
such that $f(0)=0$ and $f(z)=z^{-\frac{1}{2}}$ for $z\in B(1;\varepsilon)$. 
Hence $f(w)\in M_{n}(\mathcal{A})\cap pM_{n}(A)p=pM_{n}(\calA)p$ and satisfies 
$f(w)wf(w)=p$. Now put $v_{k}:=f(w)w_{k}$ so that, as above, 
\[\sum_{k=1}^{m}\Theta^{R}_{v_{k},v_{k}}=f(w)\left(\sum_{k=1}^{m}\Theta^{R}_{w_{k},w_{k}}\right)f(w)=p,\]
which proves the claim.
\end{proof}

\subsection{Derivations, pre-Morita equivalence bimodules and localisation}  \label{subsec:derivations_and_bimodules}
We now describe a general method to construct pre-Morita equivalence bimodules and 
pre-$C^{*}$-algebras from a family of densely defined derivations on a given $C^{*}$-algebra. 
For instance, the algebra $C^{\infty}(M)$, with $M$ a compact manifold, can be constructed in this way.

As a technical tool we will use the notion of differential seminorms introduced in \cite[Definition 3.1]{BC91}. 
The space $\ell^{1}(\mathbb{N})$ is an algebra in the convolution product
$f*g(n):=\sum_{k\leq n}f(k)g(n-k).$
The subspace $\ell^{1}_{+}(\mathbb{N}):=\ell^{1}(\mathbb{N},\mathbb{R}_{+})\subset\ell^{1}(\mathbb{N})$ 
inherits the coordinatewise ordering from $\mathbb{R}_{+}$ and satisfies 
$\ell^{1}_{+}(\mathbb{N})\ell^{1}_{+}(\mathbb{N})\subset \ell^{1}_{+}(\mathbb{N})$. 
Following \cite{BC91}, a differential seminorm on a subalgebra $\mathfrak{A}\subset A$ 
is a map $T:\mathfrak{A}\to \ell^{1}_{+}(\mathbb{N})$ such that $T(a)(0)\leq C\|a\|$ 
and $T(\lambda a)=|\lambda|T(a)$ and $T(ab)\leq T(a)*T(b)$.  The functional
$\int:\ell^{1}(\mathbb{N})\to \mathbb{C}$, $f\mapsto \sum_{k\in\mathbb{N}} f(k)$,
is positive and multiplicative, and if $T:\mathfrak{A}\to \ell^{1}_{+}(\mathbb{N})$ 
is a differential seminorm, then $\int T:\mathfrak{A}\to \mathbb{R}_{+}$ is a submulitiplicative seminorm.
\begin{prop}[Cf.~\cite{BC91}]
\label{prop: apply_BC}
Let $A$ be a $C^{*}$-algebra with norm $\|\cdot \|$, $\mathfrak{A}\subset A$ a dense 
$*$-subalgebra and $\{\partial_{j}:\mathfrak{A}\to \mathfrak{A}\}_{j=1}^{d}$ 
a finite family of commuting $*$-derivations. Then for $\alpha \in \mathbb{N}^d$ a multi-index,
$$
   \|a\|_n = \sum_{|\alpha|\leq n}\frac{\big\| \partial^\alpha a \big\|}{\alpha!}, 
    \qquad  \partial^\alpha = \partial_1^{\alpha_1}\cdots \partial_d^{\alpha_d},
     \quad|\alpha|:=\sum_{k=1}^{n}\alpha_{k},\quad  \alpha!:=\prod_{k=1}^{d}(\alpha_{k}!),
$$
is a sequence of submutliplicative seminorms on $\mathfrak{A}$. 
Let $\mathcal{A}_{n}$ denote the closure of $\mathfrak{A}$ 
in the seminorms $\|\cdot \|_{k}$, $k\leq n$ and 
$\A_{\infty}:=\varprojlim \mathcal{A}_{n}$ the Fr\'echet closure of $\mathfrak{A}$ 
in all the seminorms $\|\cdot \|_{n}$. Then for $n=0,\cdots,\infty$, $\A_{n}$ is a pre-$C^{*}$-algebra.
\end{prop}
\begin{proof}
Defining 
$\begin{pmatrix}\alpha\\ \beta \end{pmatrix}:=\prod_{k=1}^{d}\begin{pmatrix}\alpha_{k}\\\beta_{k}\end{pmatrix}$, 
one proves that 
$\partial^{\alpha}(ab)=\sum_{\beta\leq\alpha}\begin{pmatrix}\alpha\\ \beta \end{pmatrix}\partial^{\beta}(a)\partial^{\alpha-\beta}(b),$
by induction on $\alpha$. Following \cite{BC91}, for $n\in\mathbb{N}$ we consider the maps
\[T_{n}:\mathfrak{A}\to \ell^{1}_{+}(\mathbb{N}),\qquad 
T_{n}(a)(k):=\frac{\sum_{|\alpha|=k}\|\partial^{\alpha}a\|}{\alpha!} \, \text{ for }\, k\leq n, \quad 
T_{n}(a)(k)=0 \,\text{ for }\,  k>n. \]
Indeed the $T_{n}$ satisfy $T_{n}(a)(0)=\|a\|$ and $T_{n}(\lambda a)=|\lambda|T_{n}(a)$ as well as
\begin{align*}
T_{n}(ab)(k)&=\sum_{|\alpha|=k}\frac{1}{\alpha!}\big\|\partial^{\alpha}(ab)\big\|
\leq \sum_{|\alpha|= k}\sum_{\beta\leq \alpha} \frac{1}{\beta!}\big\|\partial^{\beta}(a)\big\|
  \frac{1}{(\alpha-\beta)!}\big\|\partial^{\alpha-\beta}(b)\big\| \\
&\leq \sum_{|\alpha|+|\beta|=k} \frac{1}{\beta!}\big\|\partial^{\beta}(a)\big\| \, \frac{1}{\alpha!}\big\|\partial^{\alpha}(b)\big\|
\leq T_{n}(a)*T_{n}(b)(k),
\end{align*} 
which shows that $T_{n}$ is a differential algebra norm. Moreover for $k\geq 1$ 
\begin{align*}
T_{n}(ab)(k)&\leq \sum_{|\alpha|+|\beta|=k} \frac{1}{\beta!}\big\|\partial^{\beta}(a)\big\| \, 
   \frac{1}{\alpha!}\big\|\partial^{\alpha}(b)\big\|\\
&=\|a\|\Big(\sum_{|\beta|=k}\frac{1}{\beta!} \big\|\partial^{\beta}(b) \big\|\Big)+
  \Big(\sum_{|\alpha|=k}\frac{1}{\alpha!}\big\|\partial^{\alpha}(a)\big\|\Big)\|b\|+
\sum_{ |\alpha|,|\beta|\geq 1} \frac{1}{\beta!}\big\|\partial^{\beta}(a)\big\| \, \frac{1}{\alpha!}\big\|\partial^{\alpha}(b)\big\|,
\end{align*}
so that each $T_{n}$ is of logarthmic order $\leq 1$ (see \cite[Definition 3.4]{BC91}). 
We then have $\|a\|_{n}=\int T_{n}(a)$ and the result for $\calA_n$ follows 
from \cite[Propositions 3.3 and 3.12]{BC91}. The result for $\A_{\infty}$ then follows 
since $\A_{\infty}=\bigcap_{n=0}^{\infty} \calA_{n}$ and the result holds for each $\calA_{n}$.
\end{proof}

Now let $\mathfrak{A}\subset A$ and $\mathfrak{B}\subset B$ be dense $*$-subalgebras and 
$\mathfrak{E}$ an $\mathfrak{A}$-$\mathfrak{B}$ pre-Morita equivalence bimodule. 
Suppose we are given a families  of commuting $*$-derivations 
\[ 
\{\partial_j^{L}\}_{j=1}^d:\mathfrak{A}\to\mathfrak{A}, \qquad 
\{\partial_j^{R}\}_{j=1}^d:\mathfrak{B}\to \mathfrak{B},
\] 
as well as a commuting family of operators $\nabla_{j}:\mathfrak{E}\to \mathfrak{E}$ such that for each 
$j$ and all $a\in \mathfrak{A},\xi,\eta\in \mathfrak{E}$ and $b\in \mathfrak{B}$ we have
\begin{align}
  \nabla_{j}(a\cdot \xi \cdot b)&=\partial^{L}_{j} (a)\cdot \xi\cdot b+a\cdot \nabla_{j}(\xi)\cdot b + a\cdot \xi\cdot\partial^{R}_{j}(b),    \label{eq: bimodder}   \\
   \partial^{R}_{j}(\xi\mid\eta)_{\mathfrak{B}}&=(\xi\mid\nabla_{j}(\eta))_{\mathfrak{B}}-(\nabla_{j}(\xi)\mid\eta)_{\mathfrak{B}}.   \label{eq: leftrighther}
\end{align}
It is worth noting that the identity
\begin{align}
\label{eq: lefther}
\partial^{L}_{j}{}_{\mathfrak{A}}(\xi\mid\eta)={}_{\mathfrak{A}}(\xi\mid\nabla_{j}(\eta))-{}_{\mathfrak{A}}(\nabla_{j}(\xi)\mid \eta),
\end{align}
is satisfied as well. By using Equations \eqref{eq: bimodder}, \eqref{eq: leftrighther} 
and the compatibility of left and right inner products, we find for $\xi,\eta,e\in\mathcal{E}$:
\begin{align*}
{}_{\mathfrak{A}}(\xi\mid\nabla_{j}(\eta))\cdot e-{}_{\mathfrak{A}}(\nabla_{j}(\xi)\mid \eta)\cdot e &=\xi{}\cdot (\nabla_{j}(\eta)\mid e)_{\mathfrak{B}}-\nabla_{j}(\xi)\cdot (\eta\mid e)_{\mathfrak{B}}\\
&=\xi{}\cdot (\eta\mid \nabla_{j}(e))_{\mathfrak{B}}-\xi\cdot \partial^{R}_{j}(\eta\mid e)-\nabla_{j}(\xi)\cdot (\eta\mid e)_{\mathfrak{B}}\\
&=\xi{}\cdot (\eta\mid \nabla_{j}(e))_{\mathfrak{B}}-\nabla_{j}(\xi\cdot (\eta\mid e)_{\mathfrak{B}})\\
&={}_{\mathfrak{A}}(\xi\mid \eta)\cdot \nabla_{j}(e)-\nabla_{j}(_{\mathfrak{A}}(\xi\mid \eta)\cdot e))\\
&=\partial^{L}_{j}{}_{\mathfrak{A}}(\xi\mid\eta)\cdot e,
\end{align*}
so that \eqref{eq: lefther} follows.

We write $\calA_{n},\calB_{n}$ for the pre-$C^{*}$-algebra completions obtained through Proposition \ref{prop: apply_BC}.
Similarly we write
\[
\nabla^{\alpha}:=\nabla_1^{\alpha_1}\cdots \nabla_d^{\alpha_d}, \quad \alpha \in \mathbb{N}^d,\quad \|e\|_{n}:=\sum_{|\alpha| \leq n}\frac{\|\nabla^{\alpha}(e)\|}{\alpha!}
\]
as well as $\mathcal{E}_{n}$ for the completion of $\mathfrak{E}$ in the seminorms 
up to degree $n$, and $\calE=\calE_{\infty}$ for the completion of $\mathfrak{E}$ in all these seminorms. 
\begin{defn} \label{def:HS_localisation}
Suppose that $E_B$ is a right $C^*$-module with $B$ a unital $C^*$-algebra and 
$\omega_B: B\to \C$ a  state. The localisation $\mathfrak{h}_\omega$ 
is the Hilbert space that comes from the 
completion of $E$ in the inner product 
$$
   \langle e_1, e_2 \rangle_\omega =  \omega_B \big( ( e_1 \mid e_2)_B \big).
$$
\end{defn}

\begin{remark}
If ${}_A E_B$ is a Morita equivalence bimodule and the state $\tau: B\to \C$ is 
  a trace, then there is a canonical dual tracial weight $\Tr_\tau$ on $\mathbb{K}_B(E) \cong A$ such that 
    $\Tr_\tau( \Theta_{e_1,e_2}^R) = \tau( (e_2\mid e_1)_B )$ for any $e_1, e_2 \in E$. 
    We can localise ${}_A E_B$ with a left-linear inner-product from $\Tr_\tau$ to obtain the dual 
    localisation space $\frakh_\omega^*$. Hence in this special case, the localisation Hilbert 
    space directly agrees with the Gabor analysis literature. See~\cite{AJL19} for more details.
\end{remark}

Given any state $\omega: B \to \C$, the seminorms $\|\cdot \|_{n}$ induce a family of seminorms 
$\|\cdot \|_{n, \omega}$ on the Hilbert space localisation $\mathfrak{h}_\omega$,
$$
  \| \xi \|_{n, \omega} :=\sum_{|\alpha| \leq n}
     \frac{ \|\nabla^\alpha(\xi) \|_{\frakh_\omega}  }{\alpha!}, \quad n \in \mathbb{N}.
$$

\begin{prop} 
\label{lem:pre-Morita-Frechet-HSpace}
For $n=0,\cdots , \infty$
the space $\calE_{n}$ is a $\calA_{n}$-$\calB_{n}$ pre-Morita equivalence bimodule.
Moreover, for any state $\omega:B\to \mathbb{C}$, 
the image of $\calE_{n}$ in $\frakh_\omega$ is bounded in the seminorms $\| \cdot \|_{k, \omega}$ 
for $0\leq k \leq n$.
\end{prop}
\begin{proof}
The space 
\[\mathcal{L}(\mathfrak{E}):=\left\{\begin{pmatrix}a & \xi \\ \eta^{*} & b\end{pmatrix}: a\in\mathfrak{A},\quad b\in\mathfrak{B},\quad \xi,\eta\in \mathfrak{E}\right\}\]
with multiplication and involution
\[\begin{pmatrix}a_{1} & \xi_{1} \\ \eta^{*}_{1} & b_{1}\end{pmatrix}\cdot \begin{pmatrix}a_{2} & \xi_{2} \\ \eta^{*}_{2} & b_{2}\end{pmatrix}:=\begin{pmatrix}a_{1}a_{2}+{}_\mathfrak{A}(\xi_{1}\mid \eta_{2}) & a_{1}\xi_{2}+\xi_{1}b_{2} \\ (a_{2}^{*}\eta_{1})^{*}+(\eta_{2}b_{2})^{*} & (\eta_{1}\mid \xi_{2})_{\mathfrak{B}}+b_{1}b_{2}\end{pmatrix},
\quad \begin{pmatrix}a & \xi \\ \eta^{*} & b\end{pmatrix}^{*}:=\begin{pmatrix}a^{*} & \eta^{*} \\ \xi & b^{*}\end{pmatrix},\]
is an associative $*$-algebra, the linking algebra of $\mathfrak{E}$. It is a dense $*$-subalgebra of the 
linking $C^{*}$-algebra $\mathcal{L}(E)$ of the $C^{*}$-module closure $E$ of $\mathfrak{E}$.
Using the identities \eqref{eq: bimodder}, \eqref{eq: leftrighther} and \eqref{eq: lefther}, 
one shows that the maps $\nabla_{j}$ induce a family of commuting $*$-derivations on the linking algebra via
\[\Delta_{j}\begin{pmatrix} a & \xi \\ \eta^{*} & b\end{pmatrix}:=\begin{pmatrix} \partial^{L}_{j}a & \nabla_{j}( \xi )\\ -\nabla_{j}(\eta)^{*} & \partial^{R}_{j}b\end{pmatrix},
\]
and the norms $\|e\|_{n}$ defined above are the restrictions of the norms obtained from the derivations $\Delta_{j}$.
This proves that $\calE_{n}$ is an $\calA_{n}$-$\calB_{n}$ bimodule. 
 Since $\calE_{n}$ is a subspace of $E$, the properties of Definition \ref{def:pre_MEB} 
 will follow once we show that the left and right inner products on ${}_A E_{B}$ take values in 
 $\calA_{n}$ and $\calB_{n}$, respectively, when restricted to $\calE_{n}$. 
 This in turn follows since the inner products are realised as multiplications in the linking algebra. 
 Lastly we have $\|e\|_{k,\omega}\leq \|e\|_{k}$ for any state $\omega$,
 so the image of $\calE_{n}$ is bounded 
 in each of the seminorms $\|e\|_{k,\omega}$ with $k\leq n$.
\end{proof}

\begin{remark}
The pre-Morita equivalence bimodule ${}_{\calA_\infty}{\calE_\infty}_{\calB_\infty}$ is the `smoothest' bimodule over 
pre-$C^*$-algebras that we can consider from the derivations $\{\partial_j^L\}_{j=1}^d$ and 
$\{\partial_j^R\}_{j=1}^d$. The lower-order pre-Morita equivalence bimodules 
${}_{\calA_n}{\calE_n}_{\calB_n}$ for $1\leq n < \infty$ will allow us to consider $C^*$-module 
and Hilbert space frames of differing regularity. This will be of use to us when considering the localisation 
dichotomy of Wannier bases in Section \ref{subsec:Loc_dichotomy}.
\end{remark}

\subsection{Groupoid $C^*$-algebras and equivalence}
Groupoids provide a useful generalisation of groups and group actions on spaces. 
A standard reference for groupoid  $C^*$-algebras is~\cite{Renault80}.

\begin{defn}
A groupoid is a set $\calG$ with a subset $\calG^{(2)} \subset \calG\times\calG$, 
a multiplication map $\calG^{(2)}\to\calG$, $(\gamma,\xi)\mapsto \gamma\xi$ 
and an inverse $\calG \to \calG$ $\gamma\mapsto \gamma^{-1}$ such that
\begin{enumerate}
  \item[(i)] $(\gamma^{-1})^{-1}=\gamma$ for all $\gamma\in\calG$, 
  \item[(ii)] if $(\gamma,\xi), (\xi,\eta)\in \calG^{(2)}$, then $(\gamma\xi,\eta), (\gamma,\xi\eta)\in\calG^{(2)}$, 
  \item[(iii)] $(\gamma,\gamma^{-1})\in\calG^{(2)}$ for all $\gamma\in\calG$, 
  \item[(iv)] for all $(\gamma,\xi)\in\calG^{(2)}$, $(\gamma\xi)\xi^{-1}=\gamma$ and $\gamma^{-1}(\gamma\xi) = \xi$.
\end{enumerate}
\end{defn}
Given a groupoid we denote by $\calG^{(0)}= \{\gamma \gamma^{-1}\,:\, \gamma\in\calG\}$ the space of 
units and define the source and range maps $r,s:\calG \to\calG^{(0)}$ by the equations
\begin{align*}
  r(\gamma) = \gamma\gamma^{-1},  &&s(\gamma) = \gamma^{-1}\gamma
\end{align*}
for all $\gamma\in\calG$. The source and range maps allow us to characterise
$$
   \calG^{(2)} = \big\{ (\gamma,\xi) \in \calG\times\calG \,:\, s(\gamma) = r(\xi) \big\}.
$$
Throughout this work, we will assume that $\calG$ is equipped with second countable, locally compact and 
Hausdorff topology such that the mulitplication, inversion, source and range maps are all continuous. 
A groupoid $\calG$ is \'{e}tale if the range map $r:\calG \to \calG^{(0)}$ is a local homeomorphism.
\'{E}tale groupoids have the useful property that for all $x \in \calG^{(0)}$, the fibres $r^{-1}(x)$ and 
$s^{-1}(x)$ are discrete.

\begin{examples}
\begin{enumerate}
  \item[(i)] Let $G$ be a group, then it is also a groupoid such that $G^{(0)} = \{e\}$ with multiplication 
  and inverse given by the group operation. If $G$ is discrete, then it is \'{e}tale as a groupoid.
  \item[(ii)] Let $X$ be a locally compact Hausdorff space, $G$ a locally compact group and suppose 
  there is a continuous left-action $G \to \mathrm{Homeo}(X)$ so that 
  $(g,x) \mapsto g\cdot x$ is jointly continuous.
   We can define the locally compact and Hausdorff 
  transformation groupoid $X \rtimes G$ given by pairs $(x,g) \in X \times G$ such that 
  $(X\rtimes G)^{(0)}=X$,
  \begin{align*}
  &(x,g)^{-1} = (g^{-1}\cdot x, g^{-1}), &&(x,g)(g^{-1}\cdot x, h) = (x, gh), 
  &&s(x,g) = g^{-1}\cdot x,   &&r(x,g) = x.
  \end{align*}
\end{enumerate}
\end{examples}

\begin{defn}
Let $\calG$ be a locally compact and Hausdorff groupoid. A continuous 
map $\sigma:\calG^{(2)}\to \T$ is a $2$-cocycle if 
$$
  \sigma(\gamma,\xi) \sigma(\gamma\xi,\eta) 
    = \sigma(\gamma, \xi\eta) \sigma(\xi,\eta)
$$
for any $(\gamma,\xi),(\xi,\eta)\in\calG^{(2)}$, 
and
$$
  \sigma(\gamma, s(\gamma)) = 1 = \sigma(r(\gamma),\gamma)
$$
for all $\gamma\in\calG$. We will call a groupoid $2$-cocycle normalised 
if $\sigma(\gamma,\gamma^{-1})=1$ for all $\gamma\in\calG$.
\end{defn}

 The following result is well-known, though for completeness we provide a proof.
\begin{lemma} \label{lem:2cocycle_identities}
If a groupoid $2$-cocycle $\sigma: \calG^{(2)}\to \mathbb{T}$ is normalised, then 
$\sigma(\gamma,\xi) = \ol{\sigma(\gamma\xi,\xi^{-1})}$ and 
$\sigma(\gamma,\xi) = \ol{\sigma(\xi^{-1}, \gamma^{-1})}$ for all $(\gamma,\xi) \in \calG^{(2)}$.
\end{lemma}
\begin{proof}
Using the $2$-cocyle identity,
$$
  \sigma(\gamma,\xi) \sigma(\gamma\xi, \xi^{-1} ) = \sigma(\gamma, s(\xi^{-1}) ) \sigma(\xi,\xi^{-1}) = 1
$$
so $\sigma(\gamma,\xi) = \ol{\sigma(\gamma\xi, \xi^{-1} )}$. Next we compute that 
$$
  \sigma(\gamma\xi, \xi^{-1})\sigma(\gamma\xi\xi^{-1},\gamma^{-1}) 
  = \sigma(\gamma\xi, \xi^{-1}\gamma^{-1}) \sigma(\xi^{-1},\gamma^{-1}) = \sigma(\xi^{-1},\gamma^{-1})
$$
and using the first identity
$$
  \sigma(\gamma\xi, \xi^{-1})\sigma(\gamma\xi\xi^{-1},\gamma^{-1}) 
  = \ol{\sigma(\gamma,\xi)  \sigma(\gamma\xi \xi^{-1} \gamma^{-1}, \gamma)} 
  = \ol{\sigma(\gamma, \xi) \sigma( r(\gamma\xi),\gamma)} = \ol{\sigma(\gamma,\xi)}
$$
which gives that $\ol{\sigma(\gamma,\xi)} = \sigma(\xi^{-1},\gamma^{-1})$.
\end{proof}

We briefly review the construction of groupoid $C^*$-algebras. 
\begin{defn}
A Haar system on a locally compact Hausdorff groupoid $\calG$ is a set of measures 
$\{\nu^x \,:\, x\in \calG^{(0)}\}$ on $\calG$ such that $\mathrm{supp}(\nu^x) = r^{-1}(x)$ and 
for all $f \in C_c(\calG)$,
\begin{align*}
  &\int_\calG f(\xi) \, \mathrm{d}\nu^{r(\eta)}(\xi) = \int_\calG f(\eta\xi) \, \mathrm{d}\nu^{s(\eta)}(\xi), 
  &&g(x) := \int_\calG f(\xi) \, \mathrm{d}\nu^x(\xi) \in C(\calG^{(0)}).
\end{align*}
\end{defn}

\'{E}tale groupoids always have a Haar system given by the counting measure on the (discrete) fibres 
$r^{-1}(x)$. Given $\calG$ with a $2$-cocycle $\sigma$ and Haar system $\{\nu^x\}_{x\in \calG^{(0)}}$, 
we can define a $\ast$-algebra structure on $C_c(\calG,\sigma)$, 
\begin{align*}
  &(f_1\ast f_2)(\eta) = \int_\calG f(\xi) g(\xi^{-1}\eta) \, \sigma(\xi, \xi^{-1}\eta) \,\mathrm{d}\nu^{r(\eta)}(\xi), 
  &&f^*(\xi) = \sigma(\xi, \xi^{-1}) \, \ol{f(\xi^{-1})}.
\end{align*}
We will restrict ourselves to normalised cocycles, $\sigma(\xi,\xi^{-1})=1$ as it covers all examples of interest to 
us. See~\cite{Renault80} for the general construction. 
The algebra $C_c(\calG,\sigma)$ has a right $C_0(\calG^{(0)})$-module structure, where 
$(f\cdot g)(\xi) = f(\xi) g(s(\xi))$ for $f\in C_c(\calG,\sigma)$ and 
$g \in C_0(\calG^{(0)})$. We can define a $C_0(\calG^{(0)})$-valued inner-product
$$
  (f_1 \mid f_2)_{\calG^{(0)}}(x) = \int_\calG \ol{f_1(\xi^{-1})} f_2(\xi^{-1}) \,\sigma(\xi,\xi^{-1})\, \mathrm{d}\nu^x(\xi) 
  = \int_\calG \ol{f_1(\xi^{-1})} f_2(\xi^{-1}) \, \mathrm{d}\nu^x(\xi) 
$$
as $\sigma$ is normalised. Completing this space in the norm of $C_0(\calG^{(0)})$ gives a 
right-$C_0(\calG^{(0)})$-module, which we denote by $L^2(\calG, \nu)_{\calG^{(0)}}$. 
There is a canonical left-action of $C_c(\calG,\sigma)$ on $L^2(\calG, \nu)_{\calG^{(0)}}$ by 
the (twisted) convolution product.

\begin{defn}[cf.~\cite{KS02}]
The reduced groupoid $C^*$-algebra $C^*_r(\calG,\sigma)$ is the completion of 
$C_c(\calG,\sigma)$ in the norm inherited from the embedding 
$C_c(\calG,\sigma) \hookrightarrow \End^*_{C_0(\calG^{(0)})}( L^2(\calG,\nu))$.
\end{defn}

If there is a topological space $Z$ with (continuous)
map $\rho: Z \to \calG^{(0)}$, we denote by $\calG \ltimes_\rho Z$ and $Z \rtimes_\rho \calG$ the pullback 
with respect to the 	source and range maps respectively.

\begin{defn}[$\calG$-space]
A Hausdorff topological space $Z$ is a left $\calG$-space if there is a continuous map, called the anchor  or moment map, 
$\rho:Z \to \calG^{(0)}$ and a continuous map 
$$
   \calG \ltimes_\rho Z \to Z, \qquad (\gamma, z) \mapsto \gamma \cdot z \in Z
$$
such that for $(\gamma,z) \in \calG \ltimes_\rho Z$ and $(\gamma_1,\gamma_2) \in \calG^{(2)}$,
\begin{align*}
  &\rho(\gamma\cdot z) = r(\gamma), 
  &&\rho(z) \cdot z = z, 
  &&(\gamma_1\gamma_2)\cdot z = \gamma_1 \cdot( \gamma_2\cdot z)
\end{align*}
\end{defn}

Unless otherwise stated, we will always assume that the moment map $\rho:Z\to \calG^{(0)}$ is 
open and surjective.
One may also consider a right $\calH$-space from $\phi:Z \to \calH^{(0)}$, where the 
definition is analogous to the above 
but instead using a map $Z \rtimes_\phi \calH  \to Z$ such that 
$\rho(z \cdot \eta) =  s(\eta)$.
 When the context is clear, we will write left/right-actions 
as $\gamma z$ or $z\eta$.

We say that $Z$ is a proper $\calG$-space (or that $\calG$ acts properly) if the map 
$$
  \calG \ltimes_\rho Z \to Z \times Z, \qquad (\gamma,z) \mapsto (\gamma\cdot z, z)
$$
is proper. If the map $(\gamma,z) \mapsto (\gamma\cdot z, z)$ is injective, 
then we say that 
$Z$ is a free $\calG$-space or that $\calG$ acts freely.

\begin{defn}
Let $\calG$ and $\calH$ be (locally compact, Hausdorff) groupoids and assume that 
$Z$ carries both a left-$\calG$ and right-$\calH$ action via moment
maps $\rho:Z\to \calG^{(0)}$ and $\phi: Z\to \calH^{(0)}$. We say that 
$Z$ is a $\calG$--$\calH$-bibundle if the actions commute, i.e., 
\begin{enumerate}
  \item for all $(\gamma,z) \in \calG \ltimes_\rho Z$ and $(z,\eta) \in Z \rtimes_\phi \calH$, 
  $(\gamma\cdot z) \cdot \eta = \gamma \cdot (z\cdot \eta)$,
  \item for all $(z,\eta) \in Z\rtimes_\phi \calH, \rho(z\cdot \eta) = \rho(z)$,
  \item for all $(\gamma,z) \in \calG\ltimes_\rho Z$, $\phi(\gamma\cdot z) = \phi(z)$.
\end{enumerate}

A $\calG$--$\calH$ bibundle is a groupoid equivalence if the maps 
\begin{align*}
  &\calG \ltimes_\rho Z \to Z \ast_{\calH^{(0)}} Z, \qquad \qquad (\gamma,z)\mapsto (\gamma\cdot  z, z), \\
  &Z \rtimes_\phi \calH \to Z \ast_{\calG^{(0)}} Z, \qquad \qquad (z, \eta) \mapsto (z, z\cdot \eta)
\end{align*}
are homeomorphisms.
\end{defn}

At the level of operator algebras, groupoid equivalence implies Morita equivalence of the 
(full or reduced) groupoid $C^*$-algebras~\cite{MRW, SimsWilliams}. 
Let $\calG \xleftarrow{\rho} Z \xrightarrow{\phi} \calH$ be a groupoid 
equivalence 
such that $\calG$ and $\calH$ have  Haar systems $\{\nu^{x}\}_{x\in\calG^{(0)}}$
and $\{\lambda^{y}\}_{y \in \calH^{(0)}}$ respectively.
Then there is a left (resp. right) 
action of $C_c(\calG)$ (resp. $C_c(\calH)$) on $C_c(Z)$ given by 
\begin{align}  \label{eq:gpoid_equiv_mod_structure_1}
      (f \cdot \xi)(z) &=  \int_{\calG} f(\gamma) \xi(\gamma^{-1} z) \,\mathrm{d}\nu^{\rho(z)}(\gamma),  
        \quad \xi \in C_c(Z), \,\, f \in C_c(\calG),    \nonumber  \\
    (\xi \cdot g)(z) &=  \int_{\calH} \xi( z \eta) g(\eta^{-1} ) \, \mathrm{d}\lambda^{\phi(z)}(\eta), 
        \quad \xi \in C_c(Z), \,\, g \in C_c(\calH).
\end{align}
There is also the $C_c(\calH)$-valued inner product 
\begin{equation}  \label{eq:gpoid_equiv_mod_structure_2}
   ( \xi_1 \mid \xi_2)_\calH(\eta) = \int_{\calG} \ol{ \xi_1(\gamma^{-1}z)} \xi_2 (\gamma^{-1} z\eta) \,
   \mathrm{d}\nu^{\rho(z)}(\gamma)
\end{equation}
where $z\in Z$ is chosen such that $\phi(z) = r(\eta)$. 
\begin{prop}[\cite{SimsWilliams}, Theorem 4.1] \label{prop:untwisted_gpoid_morita_equiv}
Let $Z$ be a $\calG$--$\calH$ groupoid equivalence. Then $C_{c}(Z)$ is a pre-Morita equivalence bimodule for $C_{c}(\calG)$ and $C_{c}(\calH)$. Consequently, 
$C^*_r(\calG)$ and $C^*_r(\calH)$ are Morita equivalent.
\end{prop}
We denote by ${}_{\calG} L^2(Z)_\calH$ the Morita equivalence 
bimodule that links $C^*_r(\calG)$ and $C^*_r(\calH)$. 
Because we work with reduced groupoid $C^*$-algebras, the Morita equivalence bimodule 
completion of $C_c(Z)$
is constructed from the linking groupoid 
$L = \calG \sqcup Z \sqcup Z^\mathrm{op} \sqcup \calH$~\cite{SimsWilliams}. In Section \ref{sec: twistedmorita} we discuss instances of twisted groupoid Morita equivalence.

\section{Frames of translates and Wannier bases from groupoid equivalences} \label{sec:Gpoid_Gabor_General}

We let $\calG \xleftarrow{\rho} Z \xrightarrow{\phi} \calH$ be a 
$\calG$--$\calH$ equivalence of locally compact, 
second countable and Hausdorff groupoids such that $\calH^{(0)}$ is compact and 
$\calH$ is \'{e}tale. In particular, this implies that $C^*_r(\calH)$ is  unital. 
We assume that $\calG$ has a Haar system $\{\nu^{x}\}_{x\in\calG^{(0)}}$ 
and so ${}_{C_c(\calG)}{C_c(Z)}_{C_c(\calH)}$ is a pre-Morita equivalence bimodule 
 described by Equations \eqref{eq:gpoid_equiv_mod_structure_1} 
and \eqref{eq:gpoid_equiv_mod_structure_2}. Because $\calH$ is \'{e}tale, the 
right-action in Equation \eqref{eq:gpoid_equiv_mod_structure_1} reduces to a sum 
over the discrete set $r^{-1}(\phi(z))$.

\subsection{$C^*$-module frame}

Because all $C^*$-algebras are separable, there exists a countable right module frame for 
${}_\calG L^2(Z)_\calH$. Furthermore, since $C^*_r(\calH)$ is unital, 
${}_\calG L^2(Z)_\calH$ is finitely generated and projective as a left $C^*_r(\calG)$-module 
and so has a finite left $C^*$-module frame.

A right module frame for the submodule $C_c(Z)$ is constructed 
in~\cite[Proposition 2.10]{MRW}. We briefly review this construction. We say that 
a subset $L \subset \calG$ is $r$-relatively compact if $L \cap r^{-1}(K)$ is relatively 
compact for every compact $K \subset \calG$. We consider a triple $(K, U, \epsilon)$ 
with $K\subset \calG^{(0)}$ compact with $U \subset \calG$ an $r$-relatively compact 
neighbourhood of $\calG^{(0)}$ and $\epsilon >0$. Because $\calG^{(0)}$ is 
paracompact and $\calG$ acts properly on 
$Z$, there are open, relatively compact sets $\{V_j\}_{j=1}^n \subset Z$ such that 
$\{\rho(V_j)\}_{j=1}^n$ cover $K$ and are such that $(\gamma z, z)\in V_j \times V_j$ implies 
that $\gamma \in U$. We take a partition of unity $\{\psi_j\}_{j=1}^n$ subordinate 
to the cover $\{\rho(V_j)\}_{j=1}^N$. We can then find functions $\{\tilde{\psi}_j\}_{j=1}^n$ 
such that $\mathrm{supp}(\tilde{\psi}_j) \subset V_j$ and 
$$
   \sum_{\eta \in r^{-1}(\phi(z))} \! \tilde\psi_j ( z \eta)  = \psi_j( \rho(z) ).
$$
Finally, we can approximate $\tilde\psi_j$ with $\{\varphi_j\}_{j =1}^n$ such that 
$\mathrm{supp}(\varphi_j) \subset V_j$ and 
\begin{align*}
  \Big| \tilde\psi_j(z) - \varphi_j (z) \int_\calG \varphi_j( \gamma^{-1}z ) \,\mathrm{d}\nu^{\rho(z)}(\gamma) \Big| 
  \leq \frac{\epsilon}{M}
  \qquad 
  M = \sup_z  \sum_{j=1}^n \sum_{\eta \in r^{-1}(\phi(z))}  \chi_{V_j}( z \eta).
\end{align*}
The functions $e_n = \sum_{j=1}^n {}_\calG( \varphi_j \mid \varphi_j )$ are an 
approximate identity for the left action of $C_c(\calG)$ on $C_c(Z)$. 
The functions 
$\{\varphi_j\} \subset C_c(Z)$ can then be used
 to construct a right $C^*$-module frame in the 
Morita equivalence bimodule ${}_\calG L^2(Z)_\calH$~\cite{SimsWilliams}.

\subsection{Localisation and Hilbert space frame}
For any $x \in \calH^{(0)}$ we have a state $\omega_{x}$ given by the restriction of 
$f \in C_c(\calH)$ to $\calH^{(0)}$ and then evaluation at $x \in \calH^{(0)}$. 
We see that 
$$
  \omega_{x}(g^* g) = \sum_{\eta \in r^{-1}(x)} \! | g(\eta)|^2,
$$
which shows that $\omega_{x}$ is positive. There are other possible states on 
$C^*_r(\calH)$ that one may also consider such as integration with respect to a quasi-invariant 
measure \cite{Renault80}. The reason we choose the evaluation state is because we would like to construct 
Parseval frames that have a  discrete labeling. By choosing $x \in \calH^{(0)}$, the 
discrete set $r^{-1}(x)$ provides us with such a  labeling.

\begin{lemma} \label{lem:delta_fns_on_range_fibre}
Fix an element $x\in\calH^{(0)}$. Then there are real-valued functions 
$\{\delta_\alpha\}_{\alpha \in r^{-1}(x)}\subset C_c(\calH)$ such that 
$\omega_x( \delta_{\alpha_1} \ast \delta_{\alpha_2}^* ) = \delta_{\alpha_1, \alpha_2}$.
\end{lemma}
\begin{proof}
Because $\calH^{(0)}$ is compact and $r^{-1}(x)$ is discrete, for each $\alpha \in r^{-1}(x)$ 
we take $\delta_\alpha$ a bump function supported on $U_\alpha$ such that 
$U_\alpha \cap r^{-1}(x) = \{\alpha\}$ and $\delta_\alpha(\alpha) = 1$. 
Then we compute that 
$$
  \omega_x\big(\delta_{\alpha_1} \ast \delta_{\alpha_2}^*\big) = \sum_{\beta \in r^{-1}(x)} {\delta_{\alpha_1}(\beta) }
  \delta_{\alpha_2}(\beta) = \delta_{\alpha_2}(\alpha_1) = \delta_{\alpha_1,\alpha_2}
$$
as the sum vanishes everywhere except for at most one term.
\end{proof}

\begin{lemma} Let $Z_x = \phi^{-1}(x)$ and for $z\in Z_{x}$ define the measure $\nu^{\rho(z)}$ on $Z_{x}$ by
\[\int_{Z_{x}}f(w) \mathrm{d}\nu^{\rho(z)}(w):=\int_{\mathcal{G}} f(\gamma^{-1}\cdot z)\mathrm{d}\nu^{\rho(z)}(\gamma).\]
The Hilbert space localisation $\frakh_x$ of ${}_{\calG} L^2(Z)_\calH$ in $\omega_x$ is 
$L^2(Z_x, \mathrm{d}\nu^{\rho(z)})$ with $z\in Z_{x}$ chosen arbitrarily.
\end{lemma}
\begin{proof}
We consider the inner-product, where 
\begin{align*}
  \langle e_1, e_2 \rangle_x &= \omega_x \big( (e_1\mid e_2)_\calH \big) = (e_1\mid e_2)_\calH(x) 
  = \int_{\calG} \ol{e_1(\gamma^{-1}z) }  e_2(\gamma^{-1} z x) \,\mathrm{d}\nu^{\rho(z)} (\gamma) \\
  &= \int_{\calG} \ol{e_1(\gamma^{-1}z) }  e_2(\gamma^{-1} z \phi(z)) \,\mathrm{d}\nu^{\rho(z)} (\gamma) 
  = \int_{\calG} \ol{e_1(\gamma^{-1}z) }  e_2(\gamma^{-1} z) \,\mathrm{d}\nu^{\rho(z)} (\gamma) 
\end{align*}
Hence the Hilbert space is the $L^2$-completion of $C_c(Z_x)$  with respect 
to the measure $\nu^{\rho(z)}$. Since $Z$ is an equivalence, for every $z,w\in Z_{x}$ there exists $\gamma\in r^{-1}(\rho(z))\subset \mathcal{G}$ such that $w=\gamma^{-1}z$. The measure $\nu^{\rho(z)}$ is thus independent of the choice of $z\in Z_{x}$. 
\end{proof}

Given $x\in \calH^{(0)}$ and $e \in {}_{\calG} L^2(Z)_\calH$, we let $e_x$ be the corresponding element in the 
localisation $L^2(Z_x)$. Given any $\alpha \in r^{-1}(x)$,  define a function $e_{x}^{\alpha}\in L^2(Z_x)$ by
$(e_x^\alpha)(y) = e(y\alpha)$, $y\in Z_x$. 

\begin{lemma} \label{lem:module_properties_lead_to_localisation_properties}
Let $e\in {}_{\calG} L^2(Z)_\calH$, $a\in\mathcal{H}^{(0)}$ and $\alpha\in r^{-1}(x)$. 
\begin{enumerate}
  \item[(i)] There is an equality $e_x^\alpha = ( e \cdot \delta_\alpha^*)_x$ with 
  $\delta_{\alpha}$ be the bump functions from Lemma \ref{lem:delta_fns_on_range_fibre}.
  \item[(ii)] If $e \in {}_{\calG} L^2(Z)_\calH$ is such that $(e\mid e)_\calH = 1_\calH$, then 
  $\{e_x^\alpha\}_{\alpha\in r^{-1}(x)}$ is an orthonormal system in $L^2(Z_x)$;
  \item[(iii)] Given $e_1, e_2, \xi \in {}_{\calG} L^2(Z)_\calH$, 
  $$
     \big( \Theta^R_{e_1,e_2}(\xi)\big)_x = 
       \sum_{\alpha \in r^{-1}(x)} (e_1)_x^\alpha \, \langle  (e_2)_x^\alpha, \xi_x \rangle_x.
  $$
\end{enumerate} 
\end{lemma}
\begin{proof}
We first note that $e(y\alpha)$ is well-defined as $\rho(y) = x = r(\alpha)$. For part (i) we compute 
for $y \in Z_x$,
\begin{align*}
  (e\cdot \delta_\alpha^*)(y) &= \sum_{\beta \in r^{-1}(x)} e( y\beta) \delta_\alpha^*(\beta^{-1})  
   = \sum_{\beta \in r^{-1}(x)} e(y\beta) \delta_\alpha(\beta) = e(y\alpha).
\end{align*}
Using part (i) and Lemma \ref{lem:delta_fns_on_range_fibre}, we see that 
for $e$ such that $(e\mid e)_\calH = 1$, 
\begin{align*}
  \langle e_x^{\alpha_1}, e_x^{\alpha_2} \rangle_x &= 
  \omega_x \big( (e \cdot \delta_{\alpha_1}^* \mid e\cdot \delta_{\alpha_2}^*)_\calH \big) 
  = \omega_x \big( \delta_{\alpha_1} \, (e\mid e)_\calH  \, \delta_{\alpha_2}^* \big) 
  = \delta_{\alpha_1, \alpha_2} 
\end{align*}
and so $\{e_x^\alpha\}_{\alpha\in r^{-1}(x)}$ is an orthonormal system, which proves part (ii).

For part (iii), we again compute for $y\in \phi^{-1}(x)$
\begin{align*}
    \big( \Theta^R_{e_1,e_2}(\xi)\big)(y) &= \sum_{\alpha \in r^{-1}(x)} e_1(y\alpha) (e_2\mid \xi)_\calH(\alpha^{-1}) \\
    &= \sum_{\alpha \in r^{-1}(x)} e_1(y\alpha) 
         \int_{\calG} \ol{e_2(\gamma^{-1}z)} \xi(\gamma^{-1} z \alpha^{-1})\,\mathrm{d}\nu^{\rho(z)}(\gamma).
\end{align*}
We now let $u= z\alpha^{-1}$, where 
$u\alpha = zs(\alpha) = z\phi(z) = z$ and $\rho(u) = \rho(z\alpha^{-1}) = \rho(z)$ as 
$Z$ is a groupoid equivalence. 
Hence
\begin{align*}
  \big( \Theta^R_{e_1,e_2}(\xi)\big)(y)
    &= \sum_{\alpha \in r^{-1}(x)} e_1(y\alpha)  
       \int_{\calG}  \ol{e_2(\gamma^{-1} u \alpha) } \xi( \gamma^{-1} u) \,\mathrm{d}\nu^{\rho(u)}(\gamma) \\
     &= \sum_{\alpha \in r^{-1}(x)}  (e_1)_x^\alpha(y) \, \langle (e_2)_x^\alpha,  \xi_x  \rangle_x .
        \qedhere
\end{align*} 
\end{proof}

For a countable set $J$ and a $C^{*}$-algebra $B$ we denote by $\ell^{2}(J, B)$ the 
standard Hilbert $C^{*}$-module of sequences indexed by $J$. That is
\[\ell^{2}(J, B) := \Big\{f:J\to B: \sum_{j\in J} f(j)^{*}f(j)<\infty\Big\},\]
where the series converges in $B$.
 We now come to our main result relating $C^*$-module frames to localised frames of translates.
\begin{thm}  \label{thm:equivalence_to_hilbert_frame}
Let $\{e_j\}_{j\in J}\subset {}_{\calG} L^2(Z)_\calH$ be a countable subset and 
$E:=\overline{\textnormal{span}_{\calH}\{e_j :j\in J \}}$ the closed $C^{*}_{r}(\mathcal{H})$ 
submodule generated by $\{e_j\}_{j\in J}$. The following are equivalent:
\begin{enumerate}
\item The sequence $\{e_{j}\}_{j\in J}$ is a right $C^*$-module frame of $E\subset {}_{\calG} L^2(Z)_\calH$;
\item For $\xi\in E$ the map $j\mapsto ( e_{j}\mid \xi)_\calH$ takes values in  
$\ell^{2}(J, C^{*}_{r}(\mathcal{H}))$ and for all $x \in \calH^{(0)}$ the set 
$\{(e_j)_x^\alpha\}_{j\in J, \alpha \in r^{-1}(x)}$ is a normalised tight frame for $\overline{E_{x}}\subset L^2(Z_x)$;
\end{enumerate}
\end{thm}
\begin{proof}
(1) $\Rightarrow $ (2): Using part (iii) of Lemma \ref{lem:module_properties_lead_to_localisation_properties}, we see that 
\begin{align*}
  \sum_{j\in J} \sum_{\alpha \in r^{-1}(x)} (e_j)_x^\alpha \, \langle  (e_j)_x^\alpha, \xi_x \rangle_x 
  = \sum_{j\in J} \big( \Theta^R_{e_j, e_j}( \xi)\big)_x = \xi_x
\end{align*}
as $(e_j)_{j\in J}$ is a right $C^*$-module frame.

(2) $\Rightarrow $ (1): In order to prove that $\{e_{j}\}_{j\in J}$ is a Hilbert $C^{*}$-module frame for $E$, 
we need to show that the map
\[v:E\to \ell^{2}(J,C^{*}_{r}(\mathcal{H})),\qquad \xi\mapsto ( e_{j} \mid \xi)_\calH, \]
satisfies $v^{*}v=1$. Note that $v$ is well-defined since we assume that 
$j\mapsto ( e_{j} \mid \xi)_\calH$ is an element of $\ell^{2}(J, C^{*}_{r}(\mathcal{H}))$, 
and $v$ is automatically adjointable with $v^{*}(b_{j}) =\sum_j e_j \cdot b_j$.

Since $\{(e_j)_x^\alpha\}_{j\in J, \alpha \in r^{-1}(x)}$ is a normalised tight frame for $\overline{E_{x}}\subset L^2(Z_x)$, 
the map
\[v_{x}: L^2(Z_x)\to \ell^{2}(J\times r^{-1}(x)),\qquad \psi \mapsto \langle  (e_{j})_{x}^{\alpha}, \psi \rangle_x, \]
is an isometry, for
\[(\psi\mid \psi)_{x}=\|\psi\|_{x}^2=\sum_{(j,\alpha)\in J\times r^{-1}(x)} \big| \langle (e_j)_x^\alpha, \psi \rangle_x \big|^{2}=\|v_{x}(\psi)\|^{2}=\langle v_{x}\psi,  v_{x}\psi\rangle_{\ell^{2}(J\times r^{-1}(x))}, \]
and it follows that $v_{x}^{*}v_{x}=1$ . The map $v_{x}^{*}$ is given by
\[v_{x}^{*}:(\lambda^{\alpha}_{j})\mapsto \sum_{j,\,\alpha} (e_{j})_{x}^{\alpha} \, \lambda_{j}^{\alpha},\]
and we find from Lemma \ref{lem:module_properties_lead_to_localisation_properties}
\[\xi_{x}=v_{x}^{*}v_{x}(\psi)=\sum _{\alpha,j} (e_j)_x^\alpha \, \langle (e_j)_x^\alpha, \xi_x \rangle_x  
 = \Big( \sum_{j}\Theta_{e_j,e_j}(\xi) \Big)_{x}=(v^{*}v(\xi))_{x}. \]
Therefore we can conclude that
\[\|\xi-v^{*}v(\xi)\|=\sup_{x}\|\xi-v^{*}v(\xi)\|_{x}=\sup_{x}\|\xi_{x}-(v^{*}v(\xi))_{x}\|_{x}=0,\]
and hence $v^{*}v=1$.
\end{proof}

\begin{remark}
Well-definedness of the map $v:E\to \ell^{2} (J, C^{*}_{r}(\mathcal{H}))$ entails 
that for all $e \in \mathcal{E}$ the series
\[\sum_{j} ( e \mid e_j)_\calH \, ( e_j \mid e)_\calH,    \]
is \emph{norm convergent} in $C^{*}_{r}(\mathcal{H})$. This is automatic when  the 
index set $J$ is finite but poses a non-trivial restriction for infinite $J$.
\end{remark}

\subsection{Finitely generated and projective modules} \label{subsec:fgp_Gabor}

Using the left-action of  $C^*_r(\calG)$ on $L^2(Z)$, we
 can define a representation of $\pi_x :C^*_r(\calG) \to \calB(\frakh_x)$, where 
$$
  \pi_x(f) \xi_x = (  f\cdot \xi)_x, \qquad f \in C^*_r(\calG), \,\, \xi \in L^2(Z).
$$
Concretely,
$$
   (\pi_x(f) \xi_x)(y) = \int_{\calG} f(\gamma) \xi(\gamma^{-1}y) \,\mathrm{d}\nu^{\rho\circ\phi^{-1}(x)}(\gamma), 
   \qquad f \in C_c(\calG), \,\, \xi \in C_c(Z).
$$

We consider projections, $p=p^*=p^2 \in M_n(C^*_r(\calG))$, which act compactly on 
$L^2(Z)^{\oplus n}$. 

\begin{prop}  \label{prop:tight_frame_on_proj_subspace}
Let $p=p^*=p^2 \in M_n(C^*_r(\calG))$. There is a finite set 
$\{v_j\}_{j=1}^n \subset p L^2(Z)^{\oplus n}$ such that for any $x\in \calH^{(0)}$, 
$\{v_1^\alpha, \ldots, v_n^\alpha\}_{\alpha \in r^{-1}(x)}$  is a normalised tight frame of 
$\pi_x(p)\frakh_x^{\oplus n}$.
\end{prop}
\begin{proof}
By Lemma \ref{lem: finite_frame} (2), there is a finite frame $\{v_j\}_{j=1}^n$ 
of $p L^2(Z)^{\oplus n}$. It is immediate that the localisation of $p L^2(Z)^{\oplus n}$ in 
$\omega_x$ is $\pi_x(p)\frakh_x^{\oplus n}$. The result then follows by the same proof as 
Theorem \ref{thm:equivalence_to_hilbert_frame}.
\end{proof}

Let us now consider the converse, i.e. given the Hilbert space frames 
$\{w_1^\alpha, \ldots, w_m^\alpha\}_{\alpha \in r^{-1}(x)}$ for $x \in \calH^{(0)}$, 
we construct a finitely generated and projective module.

\begin{prop} \label{prop:frame_to_fgp_module}
Let $W\subset {}_{\calG}L^2(Z)_\calH $ be  a closed submodule and $\{w_{1},\cdots, w_{n}\}$  
a finite subset of $_{\calG}L^{2}(Z)_{\mathcal{H}}$ such that 
$W:=\overline{\textnormal{span}_{C^{*}_{r}(\mathcal{H})}\{w_1,\cdots, w_{n}\}}$.  
Suppose that for all $x\in \calH^{(0)}$, $\{w_1^\alpha, \ldots, w_m^\alpha\}_{\alpha \in r^{-1}(x)}$ is 
a normalised tight frame of $\overline{W_{x}} \subset L^2(Z_x, \mathrm{d}\nu^{\rho(x)})$.
Then $W$ is a finitely generated and projective module over $C^*_r(\calH)$.
If, for each $x\in\mathcal{H}^{(0)}$,  $\{w_1^\alpha, \ldots, w_m^\alpha\}_{\alpha \in r^{-1}(x)}$ 
is an orthonormal basis, then $W\cong C^*_r(\calH)^{\oplus m}$.
\end{prop}
\begin{proof}
The first part of the Proposition follows immediately from Theorem \ref{thm:equivalence_to_hilbert_frame} 
and the fact that any $C^*_r(\calH)$-module with a finite $C^*$-module frame is 
finitely generated and projective.
Now suppose that $\{w_1^\alpha, \ldots, w_m^\alpha\}_{\alpha \in r^{-1}(x)}$ is an orthonormal 
basis and let $p_{jk} = ( w_j \mid w_k)_\calH \in C^*_r(\calH)$. 
For any $x \in \calH^{(0)}$, we have that 
\begin{align*}
  \delta_{j,k}\,\delta_{\alpha_1, \alpha_2} &= \langle w_j^{\alpha_1}, w_k^{\alpha_2} \rangle_x 
  = \omega_x\big( \delta_{\alpha_1} \, ( w_j \mid w_k )_\calH \, \delta_{\alpha_2}^* \big) 
  = \sum_{\beta \in r^{-1}(x)} (\delta_{\alpha_1} \ast p_{jk} )( \beta ) \delta_{\alpha_2}(\beta) \\
  &= (\delta_{\alpha_1} \ast p_{jk} ) ( \alpha_2) 
    = \sum_{\eta \in r^{-1}(x)} \delta_{\alpha_1}(\eta) p_{jk}(\eta^{-1}\alpha_2) 
  = p_{jk}(\alpha_1^{-1} \alpha_2)
\end{align*}
for all $\alpha_1,\alpha_2 \in r^{-1}(x)$ and all $j,k \in \{1,\ldots, m\}$. 
Now, for any $\eta \in \calH$, we can find some $x\in \calH^{(0)}$ and 
$\alpha, \beta \in r^{-1}(x)$ such that 
$\eta = \alpha^{-1} \beta$.
 Hence, 
$p_{jk}(\eta) = p_{jk}(\alpha^{-1}\beta) = \delta_{j,k}\,\delta_{\alpha, \beta}$. 
This implies that the matrix $p \in M_m(C^*_r(\calH))$ is the identity matrix. 
Hence $W\xrightarrow{\sim}  C^*_r(\calH)^{\oplus m}$.
\end{proof}

\begin{thm} \label{thm:general_Wannier_fgp_nonsmooth}
Let $p=p^* = p^2 \in M_n(C^*_r(\calG))$.
The finitely generated and projective module $pL^2(Z)^{\oplus n}$ 
with frame $\{v_j\}_{j=1}^m$ is isomorphic to the free module 
$C^*_r(\calH)^{\oplus m}$ if and only if for all 
$x \in \calH^{(0)}$, $\{v_1^\alpha, \ldots, v_m^\alpha\}_{\alpha \in r^{-1}(x)}$ 
is an orthonormal basis of 
$\pi_x(p)\frakh_x^{\oplus n}$.
\end{thm}
\begin{proof}
Suppose there is a unitary isomorphism 
of $C^*$-modules $\varphi: pL^2(Z)^{\oplus n} \xrightarrow{\sim} C^*_r(\calH)^{\oplus m}$.
It is clear that $\{1_j\}_{j=1}^m$ is right-frame of $C^*_r(\calH)^{\oplus m}$ and because 
$\varphi$ respects the inner-product structure 
$\{v_j\}_{j=1}^m$ is a frame of $pL^2(Z)^{\oplus n}$ with $v_j = \varphi^{-1}(1_j)$.
 We know that 
$\{v_j^\alpha\}$ is a normalised tight frame by Proposition \ref{prop:tight_frame_on_proj_subspace} and we see that 
\begin{align*}
  \langle (v_j)_x^{\alpha_1}, (v_k)_x^{\alpha_2} \rangle_x 
     &= \omega_x\big( (  v_j \cdot \delta_{\alpha_1}^* \mid v_k \cdot \delta_{\alpha_2}^* )_\calH \big) 
     = \omega_x\big(\delta_{\alpha_1} \, ( \varphi^{-1}(1_k) \mid \varphi^{-1}(1_j) )_\calH \,  \delta_{\alpha_2}^* \big) \\
     &= \delta_{j,k}\, \omega_x(\delta_{\alpha_1} \ast \delta_{\alpha_2}^*) 
     = \delta_{j,k} \, \delta_{\alpha_1, \alpha_2} .
\end{align*}
Hence the tight frame is orthonormal and so is an orthonormal basis. 
The converse statement follows from Proposition \ref{prop:frame_to_fgp_module}.
\end{proof}

\begin{remark}[$K$-theoretic interpretation]
Theorem \ref{thm:general_Wannier_fgp_nonsmooth} has an interpretation via the 
$K$-theory of the groupoid $C^*$-algebras. 
Using the $\ast$-isomorphism  $C^*_r(\calG) \cong \mathbb{K}_{C^*_r(\calH)}(L^2(Z))$, 
we can naturally consider any projection 
$p \in M_n(C^*_r(\calG))$ as a finite-rank operator on $L^2(Z)^{\oplus n}_\calH$. 
Therefore we can also consider $p$ as a projection in $M_m(C^*_r(\calH))$ for some $m$, with 
corresponding $K$-theory class $[p] \in K_0(C^*_r(\calH))$. If 
$pL^2(Z)_\calH \cong C^*_r(\calH)^{\oplus m}$, then 
$[p] = m[1] \in K_0(C^*_r(\calH))$ and the projection $p$ is trivial in reduced $K$-theory.
\end{remark}

Let us now consider the case where there are dense
pre-$C^*$-algebras $\calA \subset C^*_r(\calG)$ and 
$\calB \subset C^*_r(\calH)$, which will allow us to consider $C^*$-modules and 
Hilbert space frames with additional regularity as in Proposition \ref{lem:pre-Morita-Frechet-HSpace}. Thus suppose there are families  of commuting $*$-derivations 
\[\{\partial_j^{\calH}\}_{j=1}^d:C_c(\calH)\to C_{c}(\calH),\qquad \{\partial_j^{\calG}\}_{j=1}^d:C_c(\calG)\to C_{c}(\calG) ,\] 
as well as a family of maps $\nabla_{j}:C_{c}(Z)\to C_{c}(Z)$ such that for each $j$ 
and all $a\in C_{c}(G),\xi,\eta\in C_{c}(Z)$, Equations \eqref{eq: bimodder} and \eqref{eq: leftrighther} hold.
 Denote by $\calS_{k}(\calG)$ and $\calS_{k}(\calH)$ the degree $k$ Fr\'echet completions of 
 $C_{c}(\calG)$ and $C_c(\calH)$ in these seminorms (Proposition \ref{prop: apply_BC}), 
 and by $\calS_{k}(Z)$ the degree 
 $k$ Fr\'echet completion of $C_{c}(Z)$. We often write $\calS$ for $\calS_{\infty}$ in 
 each of these cases. By Proposition \ref{lem:pre-Morita-Frechet-HSpace}, 
 ${}_{\calS_{k}(\calG)} \calS_{k}(Z)_{\calS_{k}(\calH)}$ is a pre-Morita equivalence 
 bimodule for all $k=0,\cdots, \infty$.

\begin{thm}
Let $k=0,\cdots,\infty$ and $p = p^* = p^2 \in M_n(\calS_{k}(\calG))$. Then there is a finite frame 
$\{v_j\}_{j=1}^m$ of $p\calS_{k}(Z)^{\oplus n}_{\calS_{k}(\calH)}$ such that for all $x\in \calH^{(0)}$ the 
normalised tight frame $\{v_1^\beta,\ldots,v_m^\beta\}_{\beta\in r^{-1}(x)}$ 
of $\pi_x(p)\frakh_x^{\oplus n}$ 
is finite  with respect to the seminorms $\|\cdot\|_{l,x}$ on $\frakh_x$ for $0 \leq l \leq k$.
There is an isomorphism $p \calS_{k}(Z)^{\oplus n}_{\calS_{k}(\calH)} \cong \calS_{k}(\calH)^{\oplus m}$ if and 
only if  for all 
$x \in\calH^{(0)}$,  $\{v_1^\beta, \ldots, v_m^\beta\}_{\beta\in r^{-1}(x)}$ is an
 orthonormal basis.
\end{thm}
\begin{proof}
Because $v_j \cdot \delta_\beta^* \subset \calS_{k}(Z)$ for any $x\in \calH^{(0)}$ and 
$\beta \in r^{-1}(x)$, 
the first statement then follows from Lemma \ref{lem: finite_frame} and Proposition
\ref{lem:pre-Morita-Frechet-HSpace}. Lemma \ref{lem: finite_frame} and 
the fact that $\calS_{k}(\calH)$ is a pre-$C^*$-algebra give that $p\calS_{k}(\calH)^{\oplus n}$ is a free 
$\calS_{k}(\calH)$-module if and only if its $C^*$-completion is a free $C^*_r(\calH)$-module. 

If $p \calS_{k}(Z)^{\oplus n}_{\calS_{k}(\calH)} \cong \calS_{k}(\calH)^{\oplus m}$, then  there is a finite 
frame $\{v_1,\ldots, v_m\} \subset p\calS_{k}(Z)^{\oplus n}$ such that 
$(v_i\mid v_j)_\calB = \delta_{i,j} \, 1_{\calS_{k}(\calH)}$. By part (ii) of 
Lemma \ref{lem:module_properties_lead_to_localisation_properties} we 
therefore see that $\{v_1,\ldots,v_m^{\alpha}\}_{\alpha\in r^{-1}(x)}$ is an orthonormal system and hence must 
be an orthonormal basis. For the converse, we can use Proposition \ref{prop:frame_to_fgp_module} 
and the fact that the finite $C^*_r(\calH)$-module frame can be approximated arbitrarily well 
by a finite $\calS_{k}(\calH)$-module frame of the same size.
\end{proof}

\section{Frames of translates and Wannier bases from twisted transversals} \label{sec:Twsited_Gabor}

Here we consider our framework in the case of groupoid equivalences that come from 
abstract transversals with an additional twist by a normalised groupoid $2$-cocycle. 
In Section~\ref{Sec:Delone_application} we apply these results to the Delone transversal 
groupoid with twist coming from a magnetic field.
\subsection{Twisted Morita equivalence}
\label{sec: twistedmorita}

\begin{defn} \label{def:abstract_transversal}
A topological groupoid $\calG$ admits an abstract transversal if there 
is a closed subset $X\subset \calG^{(0)}$ such that 
\begin{enumerate}
  \item[(i)] $X$ meets 
every orbit of the $\mathcal{G}$-action on $\mathcal{G}^{(0)}$;
\item[(ii)] for the relative topologies on $X$ and 
\[\mathcal{G}_{X}:=\{\gamma\in\mathcal{G}: s(\gamma)\in X\}\subset \calG,\]
the restrictions $r:\mathcal{G}_{X}\to \calG^{(0)}$ and $s:\mathcal{G}_{X}\to X$ are open maps.
\end{enumerate}
\end{defn}

Given an abstract transversal $X\subset \calG^{(0)}$, $\calG\xleftarrow{r} \calG_X\xrightarrow{s} \calH$ is a 
$\calG$--$\calH$ groupoid equivalence for 
$\calH = \{ \gamma \in \calG_X \,:\, r(\gamma) \in X\}$, see \cite[Example 2.7]{MRW}. 
Examples of abstract transversals include transitive groupoids and 
groupoids from foliations.

We now fix a locally compact, second countable and Hausdorff groupoid $\calG$ such 
that $X \subset \calG^{(0)}$ is compact and admits an abstract transversal 
$\calG_X$ with $\calH$ \'{e}tale. We also fix
 a normalised groupoid $2$-cocycle $\sigma_\calG$ on $\calG$, i.e. 
$\sigma_\calG(\gamma,\gamma^{-1}) = 1$ 
for all $\gamma \in \calG$. The restriction of $\sigma_\calG$ then gives a groupoid $2$-cocycle 
$\sigma_\calH$
for $\calH$. The $2$-cocycle twists the  module structure
\begin{align*}
      (f \cdot e)(z) &=  \int_{\calG} f(\gamma) e(\gamma^{-1} z)\, \sigma_\calG(\gamma,\gamma^{-1}z) 
          \,\mathrm{d}\nu^{r(z)}(\gamma),  
        \quad e \in C_c(\calG_X), \, f \in C_c(\calG, \sigma_\calG), \\
    (e \cdot g)(z) &= \sum_{\eta \in r^{-1}(s(z))} \! e(z \eta) g(\eta^{-1}) \,\sigma_\calG(z \eta,\eta^{-1}), 
       \quad  e \in C_c(\calG_X), \, g \in C_c(\calH,\sigma_{\calH}),  \\
        ( e_1 \mid e_2)_\calH(\eta) &= \int_{\calG} \ol{ e_1(\gamma^{-1}z)} e_2 (\gamma^{-1} z\eta) 
   \,\sigma_\calG(z^{-1}\gamma, \gamma^{-1}z \eta) \,\mathrm{d}\nu^{r(z)}(\gamma), \quad r(\eta)=s(z).
\end{align*}
Proposition \ref{prop:untwisted_gpoid_morita_equiv} can be extended to the 
case of such simple twists. The case of general groupoid twists arising from $S^1$-extensions 
is handled via equivalence of 
Fell bundles, see~\cite{MT11, SimsWilliamsFell}.
\begin{prop}
The module $C_c(\calG_X)$ is a 
pre-Morita equivalence bimodule for 
$C_c(\calG,\sigma_\calG)$ and $C_c(\calH,\sigma_\calH)$. 
Consequently, $C^*_r(\calG, \sigma_\calG)$ and $C^*_r(\calH, \sigma_\calH)$ are 
Morita equivalent.
\end{prop}

\subsection{Twisted frames and Wannier bases}

We now outline the minor changes required to recover the results of Section \ref{sec:Gpoid_Gabor_General} to 
the case of twisted algebras. One advantage of restricting to normalised cocycles is that in the 
Hilbert space localisation, the inner-product simplifies. Namely, for $x\in X=\calH^{(0)}$,
\begin{align*}
  \langle (e_1)_x, (e_2)_x \rangle_x = (e_2\mid e_1)_\calH(x) 
    &= \int_\calG \ol{ e_2(\gamma^{-1} z)} e_1 (\gamma^{-1}z) \, \sigma_\calG(z^{-1}\gamma, \gamma^{-1}z) 
    \,\mathrm{d}\nu^{r\circ s^{-1}(x)}(\gamma) \\
    &= \int_\calG \ol{ e_2(\gamma^{-1} z)} e_1 (\gamma^{-1}z)  \,\mathrm{d}\nu^{r\circ s^{-1}(x)}(\gamma).
\end{align*}

\begin{lemma} \label{lem:twisted_loc_properties}
Let $x \in X$ and $\alpha \in r^{-1}(x)$. For $e_x \in \frakh_x$, define 
$e^\alpha_x(y) = e(y \alpha) \,\sigma_\calG(y\alpha,\alpha^{-1})$.
\begin{enumerate}
  \item[(i)] If $e \in L^2(\calG_X, \sigma)$ is such that $(e\mid e)_\calH = 1_\calH$, then 
  $\{e_x^\alpha\}_{\alpha \in r^{-1}(x)}$ is an orthonormal system in $\frakh_x$.
  \item[(ii)] Given $e_1, e_2, \xi \in {}_{\calG} L^2(\calG_X,\sigma)_\calH$, 
  $$
     \big( \Theta^R_{e_1,e_2}(\xi)\big)_x = 
       \sum_{\alpha \in r^{-1}(x)} (e_1)_x^\alpha \langle \xi_x, (e_2)_x^\alpha \rangle_x
  $$
\end{enumerate} 
\end{lemma}
\begin{proof}
Like in the untwisted case, we see that 
for $y \in s^{-1}(x)$,
\begin{align*}
  (e\cdot \delta_\alpha^*)(y) 
   &= \sum_{\beta \in r^{-1}(x)} e(y\beta) \delta_\alpha(\beta) \,\sigma_\calG(y\beta, \beta^{-1}) 
   = e(y\alpha)\, \sigma_\calG(y\alpha,\alpha^{-1}).
\end{align*}
Hence we can compute
\begin{align*}
  \langle e_x^{\alpha_1}, e_x^{\alpha_2} \rangle_x &= 
  \omega_x( \delta_{\alpha_1}\ast \delta_{\alpha_2}^*)  
  = \sum_{\beta \in r^{-1}(x)} \delta_{\alpha_1}(\beta) \delta_{\alpha_2}(\beta)\, \sigma_\calH(\beta,\beta^{-1}) 
  = \delta_{\alpha_1, \alpha_2}.
\end{align*}
For part (ii), can follow the same argument as Lemma \ref{lem:module_properties_lead_to_localisation_properties}.
For $y \in s^{-1}(x)$ and $u = z \alpha^{-1}$ with
$u\alpha = z s(\alpha) = z s(z) = z$ and  $r(u) = r(z\alpha^{-1}) = r(z)$,
\begin{align*}
\big( \Theta^R_{e_1,e_2}(\xi)\big)(y) &=
     \sum_{\alpha \in r^{-1}(x)} e_1(y\alpha) \, \sigma_\calG(y\alpha, \alpha^{-1})
     \int_{\calG} \ol{e_2(\gamma^{-1} u \alpha) } \xi( \gamma^{-1} u)  \,
       \sigma_\calG(\alpha^{-1}u^{-1}\gamma, \gamma^{-1}u )\,  \mathrm{d}\nu^{r(u)}(\gamma) \\
   &= \sum_{\alpha \in r^{-1}(x)} e_1(y\alpha) \, \sigma_\calG(y\alpha, \alpha^{-1})
    \int_{\calG} \ol{e_2(\gamma^{-1} u \alpha) \, \sigma_\calG(\gamma^{-1}u\alpha, \alpha^{-1}) } \, 
     \xi( \gamma^{-1} u) \, \mathrm{d}\nu^{r(u)}(\gamma) \\
     &= \sum_{\alpha \in r^{-1}(x)}  (e_1)_x^\alpha(y) \, \langle (e_2)_x^\alpha, \xi_x \rangle_x,
\end{align*} 
where we used 
the $2$-cocycle identity
$$
  \sigma_\calG(\alpha^{-1}, u^{-1}\gamma) \sigma_\calG(\alpha^{-1}u^{-1}\gamma, \gamma^{-1}u) 
   = \sigma_\calG(\alpha^{-1}, s(\gamma^{-1}u) ) \sigma_\calG(u^{-1}\gamma, \gamma^{-1} u) = 1
$$
which implies that 
$\sigma_\calG(\alpha^{-1}u^{-1}\gamma, \gamma^{-1}u) = \ol{\sigma_\calG(\alpha^{-1}, u^{-1}\gamma) }$.
Then using Lemma \ref{lem:2cocycle_identities} 
\[
  \ol{\sigma_\calG(\alpha^{-1}, u^{-1}\gamma) } = \sigma_\calG( \gamma^{-1}u, \alpha) 
  = \ol{ \sigma_\calG( \gamma^{-1}u \alpha, \alpha^{-1}) } .     \qedhere
\]
\end{proof}

Given elements $w_j^{\alpha_1}, w_k^{\alpha_2} \in \frakh_x = L^2(\calG_X, \mathrm{d}\nu^{r\circ s^{-1}(x)})$, 
we can compute that 
\begin{align*}
  \langle w_j^{\alpha_1}, w_k^{\alpha_2}  \rangle_x &= \omega_x \big( \delta_{\alpha_1} (w_j\mid w_k)_\calH  
     \delta^*_{\alpha_2} \big) 
  = \big( \delta_{\alpha_1} \ast (w_j \mid w_k)_\calH \big) (\alpha_2) \\
  &= ( w_j \mid w_k)_\calH( \alpha_1^{-1} \alpha_2) \, \sigma_\calH(\alpha_1, \alpha_1^{-1}\alpha_2).
\end{align*}
Hence, if $\langle w_j^{\alpha_1}, w_k^{\alpha_2}  \rangle_x = \delta_{j,k}\, \delta_{\alpha_1,\alpha_2}$ 
for some $\alpha_1, \alpha_2 \in r^{-1}(x)$, the $2$-cocycle term will be $1$.

At this point we can follow the same arguments as those in 
Theorem \ref{thm:equivalence_to_hilbert_frame} and Section \ref{subsec:fgp_Gabor}, so we summarise our results.

\begin{prop} \label{prop:twisted_general_loc_frame}
If $(e_j)_{j\in J}$ is a right $C^*$-module frame of ${}_{\calG}L^2(\calG_X, \sigma)_{\calH}$, then for all 
$x \in X$  the set 
$\{(e_j)_x^\alpha\}$ for $j\in J$ and $\alpha \in r^{-1}(x)$ is a normalised tight frame of $\frakh_x$. 
\end{prop}

We can again define a representation of $\pi_x :C^*_r(\calG, \sigma_\calG) \to \calB(\frakh_x)$, 
where
$$
   (\pi_x(f) \xi_x)(y) = \int_{\calG} f(\gamma) \xi(\gamma^{-1}y) \, \sigma_\calG(\gamma,\gamma^{-1}y) 
       \,\mathrm{d}\nu^{r \circ s^{-1}(x)}(\gamma), 
   \qquad f \in C_c(\calG), \,\, \xi \in C_c(\calG_X).
$$

We consider the case of pre-$C^*$-algebras $\calS_k(\calG, \sigma) \subset C^*_r(\calG,\sigma_\calG)$ and  
$\calS_{k}(\calH,\sigma) \subset C^*_r(\calH,\sigma_\calH)$ defined from families of derivations with a 
pre-Morita equivalence bimodule 
 ${}_{\calS_k(\calG,\sigma)}{\calS_{k}(\calG_X, \sigma)}_{\calS_k(\calH,\sigma)}$ defined from a 
 family of maps $\nabla_{j}:C_{c}(Z)\to C_{c}(Z)$ using the construction in 
 Proposition \ref{lem:pre-Morita-Frechet-HSpace}.

\begin{prop}
Let $k=0,\cdots,\infty$ and $p=p^*=p^2 \in M_n(\calS_k(\calG, \sigma))$. There are elements 
$\{v_j\}_{j=1}^m \subset p \calS_{k}(\calG_X, \sigma)^{\oplus n}_{\calS_k(\calH,\sigma)}$ such that
for all $x\in X$,  $\{v_1^{\beta}, \ldots , v_m^{\beta}\}_{\beta \in r^{-1}(x)}$ 
is a normalised tight frame 
of $\pi_x(p)\frakh_x^{\oplus n}$
that is finite under the seminorms $\|\cdot\|_{l, x}$ on $\frakh_x$ for $0\leq l \leq k$.
This tight frame is an orthonormal basis for all 
$x\in \calH^{(0)}$ if and only if 
$p\calS_{k}(\calG_X,\sigma)^{\oplus n}_{\calS_k(\calH,\sigma)} \cong \calS_k(\calH,\sigma)^{\oplus m}$.
\end{prop}

\section{Frames of translates and Wannier bases for the Delone groupoid} \label{Sec:Delone_application}

\subsection{Delone sets and the transversal groupoid}

We review some of the material from~\cite{BHZ00} as outlined in \cite{BMes}. 
We denote by $B(x;K) \subset \R^d$ the open ball centered at $x$ with radius $K>0$.
\begin{defn}
Let $\calL \subset\R^d$ be discrete and infinite and fix $0<r<R$.
\begin{enumerate}
  \item $\calL$ is $r$-uniformly discrete if 
    $|B(x;r)\cap \calL| \leq 1$ for all $x\in\R^d$. 
  \item $\calL$ is $R$-relatively dense if 
   $|B(x;R)\cap \calL| \geq 1$ for all $x\in\R^d$.
\end{enumerate}
An $r$-uniformly discrete and $R$-relatively dense set $\calL$ is called an 
$(r,R)$-Delone set.
\end{defn}

\begin{prop}[\cite{BBD18}, Section 3.2] \label{prop:Del_set_topology}
The set of $(r,R)$-Delone sets is a compact and metrizable space. 
Let $d_H$ denote the Hausdorff distance between sets. A
neighbourhood base at $\calL \in\mathrm{Del}_{(r,R)}$ is given by the sets
$$
   U_{\epsilon,M}(\calL) = \big\{ \calL' \in \mathrm{Del}_{(r,R)}\,:\, 
      d_H\big( \calL \cap B(0;M), \, \calL' \cap B(0;M) \big) < \epsilon \big\}, \quad M,\,\varepsilon >0.
$$ 
\end{prop}

The set of Delone sets $\mathrm{Del}_{(r,R)}$ is clearly invariant under translations and rotations.

\begin{defn}
Let $\Lambda$ be a an $(r,R)$-Delone subset of $\R^d$. The \emph{continuous hull of} $\Lambda$ is the 
dynamical system $(\Omega_{\Lambda}, \R^d, T)$, where $\Omega_{\Lambda} \subset \mathrm{Del}_{(r,R)}$ 
is the closure 
of the orbit of $\Lambda$  under the translation action. 
\end{defn}

The continuous hull of $\Lambda$ therefore gives a locally compact Hausdorff groupoid 
$\Omega_{\Lambda} \rtimes \R^d$. This groupoid admits a transversal in the sense of 
Definition \ref{def:abstract_transversal}.

\begin{defn}
The transversal of $\Lambda$ is given by the set
$$
   \Omega_0 = \{ \calL \in \Omega_{\Lambda}\,:\, 0 \in  \calL  \},
$$
\end{defn}
We see that $\Omega_0$ is a closed subset of $\Omega_{\Lambda}$ and so is 
compact.

\begin{prop}[\cite{Kellendonk97}, Lemma 2] \label{prop:etale}
Given a Delone set $\Lambda$ with transversal $\Omega_0$, define the set 
$$
  \calG_\mathrm{Del} := \big\{ (\calL, x)\in \Omega_0 \times\R^d \, :\, x \in \calL \big\},
$$
 with maps 
\begin{align*} \label{eq:derivation_properties}
  &(\calL,x)^{-1} = (\calL-x, -x), &&(\calL,x)\cdot (\calL-x, y) = (\calL,x+y), 
  &&s(\calL,x) = \calL-x, &&r(\calL,x) = \calL
\end{align*}
and unit space $\calG^{(0)}=\Omega_0$. Then $\calG_\mathrm{Del}$ is a Hausdorff \'{e}tale groupoid in the relative topology inherited from $\Omega_0\times\mathbb{R}^{d}$.
\end{prop}

\begin{notation}
Following the previous proposition, we will let $\calG_\mathrm{Del}$ denote the \'{e}tale 
groupoid from a Delone set. We let $\calF = \Omega_{\Lambda}\rtimes \R^d$ be the 
crossed product groupoid.
\end{notation}

\begin{prop}[\cite{BMes}, Proposition 2.16] \label{dtop} 
Let $\mathcal{L}\subset\mathbb{R}^{d}$ be an $(r,R)$-Delone set with transversal $\Omega_0$ and associated groupoid $\mathcal{G}_\mathrm{Del}$. For $U\subset\Omega_0$ an open set,  the sets
\begin{align*}
V_{(U,y,\varepsilon)}&:=\left(U\times B(y;\varepsilon)\right){}\cap {}\mathcal{G}_\mathrm{Del} 
     =\{(\calL, x)\in \Omega_{0}\times \mathbb{R}^{d} \,:\, \calL\in U,
      \, x \in \mathcal{L}\cap B(y;\varepsilon)\},
\end{align*}
form a base for the topology on $\mathcal{G}_\mathrm{Del}$. For $0<\varepsilon<r/2$, the restriction
$s:V_{(U,y,\varepsilon)}\to \Omega_0$ is a homeomorphism onto its image. Moreover,
 the set $\Omega_0\subset \Omega_{\Lambda}$ is an abstract transversal and 
the groupoid $\calG_\mathrm{Del}\subset \Omega_{\Lambda} \rtimes \R^d$, with the subspace topology, 
is equivalent to $\Omega_{\Lambda}\rtimes\mathbb{R}^{d}$. 
\end{prop}

Let us now fix a normalised $2$-cocycle $\sigma : (\Omega_{\Lambda} \rtimes \R^d)^{(2)} \to \T$, 
$\sigma(\gamma, \gamma^{-1}) = 1$ for all $\gamma \in \Omega_{\Lambda} \rtimes \R^d$, 
which also restricts to a $2$-cocycle on $\calG_\mathrm{Del}$.
Our main motivation to consider such twists comes from the following example.

\begin{example}[Magnetic twists, Section 2.2 of~\cite{BLM13}] \label{ex:mag_twist}
Working with the continuous hull $\Omega_{\Lambda} \rtimes \R^d$, we 
can define a 
$2$-cocycle, $\sigma: \calF^{(2)} \to \T$ as follows. 
We first define a parametrised magnetic field as a continuous map  
$B: \Omega_{\Lambda} \to \bigwedge^2 \R^d$. Then we define 
$$
  \sigma((\calL,x), (\calL-x,y)) = \exp\big( -i\Gamma_\calL \langle 0, x, x+y \rangle \big), 
  \qquad \Gamma_\calL \langle x,y,z \rangle = \int_{\langle x,y,z \rangle} \! B(\calL)
$$
and $\langle x,y,z \rangle \subset \R^{2d}$ is the triangle with corners 
$x,y,z \in \R^d$. Hence $\Gamma_\calL \langle 0, x, x+y \rangle$ measures the magnetic  
flux through the triangle defined by the points $0,x,x+y\in \R^d$. 
It is shown in~\cite{BLM13} that $\sigma$ is a well-defined 
$2$-cocycle. 
We remark that $\sigma$ will always be trivial for $d=1$ and is 
normalised because
$$
  \sigma((\calL,x),(\calL-x,-x)) = 
  \exp\big( -i\Gamma \langle 0, x, 0 \rangle \big) = 1.
$$
If the magnetic field is constant over $\Omega_{\Lambda}$, then our general flux equation 
can be described using a real-valued and skew-symmetric matrix $B$ with 
$$
   \sigma((\calL,x), (\calL-x,y)) = \exp \big( -i \langle x, B(x+y) \rangle \big) = 
   \exp \big( -i \langle x, By \rangle \big).
$$
The $2$-cocycle $\sigma$ on $\Omega_{\Lambda}\rtimes \R^d$  also restricts to 
a normalised $2$-cocycle on $\calG_\mathrm{Del}$, where we note that if 
$((\calL,x), (\calL-x,y)) \in \calG_\mathrm{Del}^{(2)}$, the points 
$0,\,x,\,x+y \in \calL$ and so $\Gamma_{\calL} \langle 0, x, x+y \rangle$ gives a flux 
through the triangle with points in $\calL$.
\end{example}

\begin{remark}
Given a groupoid $2$-cocycle $\sigma: (\Omega_{\Lambda}\rtimes \R^d)^{(2)}\to \T$, 
the twisted groupoid $C^*$-algebra 
$C^*_r(\Omega_{\Lambda}\rtimes \R^d, \sigma)$ is canonically isomorphic to the 
twisted crossed product $C(\Omega_{\Lambda}) \rtimes_{\sigma'} \R^d$, where 
$$
   \sigma':\R^d\times\R^d \to \calU(C(\Omega_{\Lambda})), \qquad 
   \sigma'(x,y) = \sigma((\calL,x), (\calL-x,y)) .
$$
\end{remark}

Regularity, smoothness and decay properties of functions on $\Omega_{\Lambda}\rtimes\mathbb{R}^{d}$ 
and $\calG_\mathrm{Del}$ are encoded via 
the groupoid $1$-cocycles 
$$
   c_k: \Omega_{\Lambda}\rtimes\mathbb{R}^{d} \to \R, \qquad c(\calL,x) = x_k, \quad k \in\{1,\ldots,d\}.
$$
and their restrictions to $\calG_{\mathrm{Del}}$. 
It is shown in~\cite[Proposition 2.17]{BMes} that the groupoid cycles are \emph{exact}, in that 
$c_k^{-1}(0)$ has a Haar system and $c_k$ is a quotient map onto its image.

Given the cocycles $c_j:\Omega_{\Lambda}\rtimes\R^d\to \R^d$, we obtain families of commuting derivations 
$\{\partial_j\}_{j=1}^d$ on both $C_c(\Omega_{\Lambda}\rtimes\R^d, \sigma)$ and $C_c(\calG_\mathrm{Del},\sigma)$ given by 
$(\partial_jf)(\mathcal{L},x) = x_j f(\mathcal{L},x)$.
For $k=0,\cdots,\infty$, we obtain pre-$C^{*}$-algebra completions 
$\mathcal{A}_{k}$ of $C_c(\Omega_{\Lambda}\rtimes\R^d, \sigma)$ and 
$\mathcal{B}_{k}$ of $C_{c}(\mathcal{G}_{\mathrm{Del}},\sigma)$ using 
Proposition \ref{prop: apply_BC}.

\subsection{The transversal $\calG_\mathrm{Del}$-space and its localisation}

Following Section \ref{sec:Twsited_Gabor}, we consider the space
$$
  \calF_{\Omega_0} := \big\{ (\calL,x)\in \Omega_{\Lambda} \rtimes \R^d\,:\, x \in \calL \big\}, \qquad 
  s:\calF_{\Omega_0} \to \Omega_0, \quad s(\calL,x) = \calL-x,
$$
which implements a groupoid equivalence between 
$\calF=\Omega_{\Lambda} \rtimes \R^d$ and $\calG_\mathrm{Del}$. 
Thus $C_c(\calF_{\Omega_0})$ is a pre-Morita equivalence bimodule for 
$C_{c}(\Omega_{\Lambda}\rtimes\mathbb{R}^{d},\sigma)$ and $C_{c}(\calG_{\mathrm{Del}},\sigma)$ 
and can be completed into the Morita equivalence bimodule 
${}_{\calF} L^2(\calF_{\Omega_0}, \sigma)_{\calG_\mathrm{Del}}$
between $C^*_r(\Omega_{\Lambda}\rtimes \R^d, \sigma)$ and  $C^*_r(\calG_\mathrm{Del}, \sigma)$. 

\begin{lemma}
The restrictions of the cocycles $c_{j}:\Omega_{\Lambda}\rtimes\mathbb{R}^{d}\to \mathbb{R}$ 
to $\mathcal{F}_{\Omega_0}$ define maps
\[\nabla_{j}:C_c(\mathcal{F}_{\Omega_0})\to C_c(\mathcal{F}_{\Omega_0}),\qquad \nabla_{j}(f)(\mathcal{L},x):=x_{j} f(\mathcal{L},x).\]
For all $a\in C_{c}(\calF,\sigma)$, $b\in C_{c}(\calG_\mathrm{Del},\sigma)$ and 
$\xi,\eta\in C_{c}(\calF_{\Omega_0})$ we have 
\begin{align*}
   \nabla_{j}(a\cdot \xi \cdot b)&=\partial_{j} (a)\cdot \xi\cdot b+a\cdot \nabla_{j}(\xi)\cdot b + a\cdot \xi\cdot\partial_{j}(b),      \\
    \partial_{j}(\xi\mid\eta)_{\mathcal{G}_{\mathrm{Del}}}&=(\xi\mid\nabla_{j}(\eta))_{\mathcal{G}_{\mathrm{Del}}}-(\nabla_{j}(\xi)\mid\eta)_{\mathcal{G}_{\mathrm{Del}}}.   
\end{align*}
Consequently, for all $k=0,\cdots , \infty$, the space $C_{c}(\mathcal{F}_{\Omega_0})$ 
can be completed into a pre-Morita equivalence bimodule 
${}_{\calA_k} \calS_{k}(\calF_{\Omega_0},\sigma)_{\calB_{k}}$ for the pre-$C^{*}$-algebras 
$\calA_{k}\subset C^*_r(\Omega_{\Lambda} \rtimes \R^d, \sigma)$ and 
$\calB_{k}\subset C^*_r(\calG_\mathrm{Del}, \sigma)$.
\end{lemma}
\begin{proof} 
As $\calF_{\Omega_0}$ and $\mathcal{G}_{\mathrm{Del}}$ are subspaces of 
$\Omega_{\Lambda}\rtimes\mathbb{R}^{d}$ and the bimodule structure and 
inner product are induced by the convolution product in 
$C_{c}(\Omega_{\Lambda}\rtimes\mathbb{R}^{d},\sigma)$, the required identities follow 
from the fact that multiplication by $x_{j}$ is a derivation of 
$C_{c}(\Omega_{\Lambda}\rtimes\mathbb{R}^{d},\sigma)$.
\end{proof}

For every $\calL \in \Omega_0$, there is a state $\omega_\calL$ on $C^*_r(\calG_\mathrm{Del}, \sigma)$ 
such that $\omega_\calL(f) = f(\calL,0)$ for $f\in C_c(\calG_\mathrm{Del}, \sigma)$. Note that 
$$
  \omega_\calL(f^*f) = \sum_{y \in \calL} |f(\calL - y, -y)|^2 , \qquad \omega_{\calL}( 1_{C^*_r(\calG_\mathrm{Del})}) = 1
$$
and so $\omega_\calL$ is faithful.
From this point, all results from Section \ref{sec:Twsited_Gabor} apply to the Delone groupoid 
setting. Though for concreteness, we highlight some key aspects of this example.

\begin{lemma}
For every $\calL \in\Omega_0$, the localised Hilbert space $\mathfrak{h}_\calL \cong L^2(\R^d)$.
\end{lemma}
\begin{proof}
We define a map $\beta_\calL \to L^2(\R^d)$ that agrees with the localised Hilbert space 
inner product. Namely, we consider $\beta_\calL: L^2(\calF_{\Omega_0}) \to L^2(\R^d)$, given by 
$[\beta_{\calL}(\xi)](x) = \xi(\calL-x, -x)$ for almost all $x$. To see why this is true, we first note 
that for any $\calL \in \Omega_0$, $s^{-1}(\calL) = \{(\calL-x,-x)\}_{x\in \R^d}$ and 
 the measure on $s^{-1}(\calL)$ is just the Lebesgue measure on $\R^d$. We also 
see that
$$
  \langle \xi_1, \xi_2 \rangle_\calL =   (\xi_1 \mid \xi_2)_{C^*_r(\calG_\mathrm{Del})}(\calL, 0) = 
   \int_{\R^d}\!  \ol{\xi_1(\calL-y,-y)} {\xi_2(\calL-y,-y)} \,\mathrm{d}y. 
$$
Hence we see there is a canonical identification of $\mathfrak{h}_\calL$ with 
$\beta_\calL[L^2(\calF_{\Omega_0})] \cong L^2(\R^d)$.
\end{proof}

For any $\xi \in {}_\calF L^2(\calF_{\Omega_0})_{\calG_\mathrm{Del}}$, 
let $\xi_\calL \in \frakh_\calL$ be its localisation. 
Because $\nabla^\alpha \xi(\calL, x) = x^\alpha \xi( \calL,x)$ for $\alpha \in \mathbb{N}^d$ and 
$\xi \in C_c(\calF_{\Omega_0})$, Proposition \ref{lem:pre-Morita-Frechet-HSpace} gives the 
following.

\begin{lemma} \label{lem:dense_localised_subspace_decay}
For any $k=0,\cdots , \infty$ and $\calL \in \Omega_0$, every element in dense subspace 
$\beta_\calL[\calS_k(\calF_{\Omega_0})]\subset\frakh_\calL$ 
has polynomial decay of at least degree $k$,
$$ 
    \| x^\alpha \xi_\calL \|_{\frakh_\calL} \leq C, \quad  \xi \in  \calS_k( \calF_{\Omega_0}), \,\, \alpha \in \N^d, 
    \,\, |\alpha| \leq k.
$$
\end{lemma}

\begin{lemma} \label{lem:bump_func_props}
Let $\chi$ be a smooth and real-valued bump-function such that $\mathrm{supp}(\chi)\subset B(0;r/2)$, 
$\chi(x)=\chi(-x)$, $\chi(0)=1$ and $\|\chi\|_2 = 1$.
\begin{enumerate}
  \item[(i)] Extend $\chi$ to a function $\chi \in C_c( \calF_{\Omega_0})$ such that 
  $\chi(\calL, x) = \chi(x)$. Then 
  $(\chi\mid \chi)_{\calG_\mathrm{Del}} = 1_{\calG_\mathrm{Del}}$.   
  \item[(ii)] Given $p \in \R^d$, define the function 
$\chi_p \in C_c(\calG_\mathrm{Del},\sigma)$ by $\chi_p(\calL,x) = \chi(x-p)$. 
Then for any $\calL \in \Omega_0$ and $p,q \in \calL$, 
$\omega_\calL( \chi_{p} \ast \chi_q^*) = \delta_{p,q}$.
\end{enumerate}
\end{lemma}
\begin{proof}
For part (i) we compute 
$$
 (\chi \mid \chi)_{\calG_\mathrm{Del}} (\calL, y) 
  = \int_{\R^d}\!  \chi(\calL-z, -z) \chi(\calL-z, y-z)\, \sigma((\calL,z),(\calL-z,y-z)) \,\mathrm{d}z.
$$
The integral will be zero unless $-z, \, y-z \in \calL - z \cap B(0;r/2)$. 
But because $y \in \calL$, $\calL \in \Omega_0$ and $\calL$ is uniformly $r$-discrete, 
this will only happen when $y = 0$. Then, because the $2$-cocycle is $1$ when $y=0$ and $\|\chi\|_2 = 1$, 
$(\chi\mid \chi)_{\calG_\mathrm{Del}} (\calL, y) = \delta_{y,0} = 1_{\calG_\mathrm{Del}}(\calL,y)$.

For part (ii) we compute that
\begin{align*}
  (\chi_{p} \ast \chi_{q}^*)(\calL,0) 
  &= \sum_{y\in\calL}  \chi_{p}(\calL,y) \ol{\chi_{q}(\calL,y)} \,\sigma((\calL,y),(\calL-y,-y)) \\
  &= \sum_{y\in \calL} \chi(y-p) \chi(y-q) = \chi(p-q) = \delta_{p,q}
\end{align*}
where we have used that $\mathrm{supp}(\chi)\subset B(0,r/2)$ and $\calL$ is $r$-discrete.
\end{proof}

Given $(\calL,y) \in \calG_\mathrm{Del}$,
define the action
$$
\xi^{(\calL,y)}(x) = \xi(\calL-x,y-x) \, \sigma((\calL-x,y-x),(\calL-y,-y)) \in \frakh_\calL, 
\quad \xi \in L^2(\calF_{\Omega_0}). 
$$
In the case that $\sigma$ is comes from a magnetic twist that is constant over the 
unit space $\Omega_{\Lambda}$, we see that 
$\xi^{(\calL,y)}(x) = e^{ -i \langle x,By \rangle} \xi( \calL-x,y-x)$ with $B$ a real-valued 
skew-adjoint matrix describing the magnetic field. Lemmas  \ref{lem:twisted_loc_properties} and  \ref{lem:bump_func_props} now give the following.
\begin{lemma} \label{lem:gpoid_trans_is_right_action}
\begin{enumerate}
  \item[(i)] Recall the functions $\{\chi_p\}$ from Lemma \ref{lem:bump_func_props}.  Then 
for any $\xi \in L^2(\calF_{\Omega_0})$ and $(\calL,p)\in\calG_\mathrm{Del}$, 
$\xi^{(\calL,p)} = ( \xi \cdot \chi_{p}^* )_\calL$.
  \item[(ii)] Let $e \in {}_{\calF}L^2(\calF_{\Omega_0},\sigma)_{\calG_\mathrm{Del}}$ 
  be such that $(e\mid e)_{C^*_r(\calG_\mathrm{Del})} = 1_{C^*_r(\calG_\mathrm{Del})}$. Then 
for any $\calL \in \Omega_0$ the set $\{e^{(\calL,y)}\}_{y\in\calL}$ is an orthonormal 
system in $\frakh_\calL$.
\end{enumerate}
\end{lemma}

\begin{prop} \label{prop:abstract_HS_frame}
Let $(e_j)_{j\in J}$ be a (countable) right frame of ${}_{\calF} L^2(\calF_{\Omega_0},\sigma)_{\calG_\mathrm{Del}}$. 
Then for all $\calL\in\Omega_0$, the set $\{e_j^{(\calL,y)}\}_{j\in J, y \in\calL}$
 is a normalised tight frame for $\mathfrak{h}_\calL \cong  L^2(\R^d)$. 
 If $(e_j)_{j\in J} \subset {}_{\calA_k} \calS_k(\calF_{\Omega_0},\sigma)_{\calB_k}$ for 
 $k=0,\cdots , \infty$, then the 
 normalised tight frame $\{e_j^{(\calL,y)}\}$  
 has polynomial decay of at least order $k$.
\end{prop}
\begin{proof}
The first statement is a special case of Proposition \ref{prop:twisted_general_loc_frame}.
Lemma \ref{lem:dense_localised_subspace_decay} ensures that the elements 
$e_j^{(\calL,0)}(x) = e_j(\calL-x,-x)$ have polynomial decay of at least degree $k$. The translation 
$e_j^{(\calL,y)}(x) = e_j(\calL-x,y-x)\, \sigma((\calL-x,y-x),(\calL-y,-y))$ 
will have the same decay properties for all $y\in\calL$.
\end{proof}

We can define an action $\pi_\calL: C(\Omega_{\Lambda}) \rtimes \R^d$ on the localisation 
Hilbert space $\frakh_\calL\cong L^2(\R^d)$ by 
$$
  (\pi_\calL(f)\xi_\calL)(x) = \big( f \cdot \xi)_\calL(x) = (f\cdot \xi)(\calL-x,-x).
$$
Explicitly, we can compute that 
\begin{align*}
  (f \cdot \xi)(\calL-x,-x) &=  \int_{\R^d} f(\calL-x,u-x) \xi_\calL(u)\, \sigma((\calL-x,u-x), (\calL-u,-u) ) \,\mathrm{d}u.
\end{align*}

Let us now consider the localisation $\pi_\calL(p)\frakh_\calL^{\oplus n}$ for 
$p \in M_n( C^*_r(\Omega_{\Lambda}\rtimes \R^d, \sigma))$ a projection.

\begin{thm} \label{thm:Delone_frame_or_basis}
Let $k =0, \cdots, \infty$ and $p=p^*=p^2 \in M_n(\calA_k)$. 
Then there are elements 
$\{e_j\}_{j=1}^m \subset p\calS_k(\calF_{\Omega_0},\sigma)_{\calB_k}^{\oplus n}$
such that for all $\calL \in \Omega_0$,  $\{e_1^{(\calL,y)}, \ldots, e_m^{(\calL,y)}\}_{y\in \calL}$ 
is a normalised tight frame of  $\pi_\calL(p) L^2(\R^d, \C^n)$ with polynomial decay of at least degree $k$. 
The normalised tight frame $\{e_1^{(\calL,y)}, \ldots, e_m^{(\calL,y)}\}_{y\in \calL}$ is 
an orthonormal basis for all $\calL \in \Omega_0$ if and only if 
$p\calS_k(\calF_{\Omega_0},\sigma)_{\calB_k}^{\oplus n} \cong \calB_k^{\oplus m}$.
\end{thm}

For concreteness, we note that the case $p \in M_n(\calA)$ with $\calA= \calA_\infty$ gives 
a normalised tight frame with faster than polynomial decay.

\begin{remark}[Invariance under homotopies of $2$-cocycles] \label{rk:mag_field_htpy}
Let us briefly consider the stability of Theorem \ref{thm:Delone_frame_or_basis} under deformations 
of the groupoid $2$-cocycle $\sigma$ using results from Gillaspy~\cite{Gillaspy15}. Given the 
groupoid $\calF = \Omega_\Lambda \rtimes \R^d$, we can consider the trivial bundle of 
groupoids 
$\calF \times [0,1]$ equipped with the product topology so that groupoid operations preserve the 
fibres and such that $\calF \times [0,1]$ is a locally compact Hausdorff groupoid. A homotopy of 
groupoid $2$-cocycles is a groupoid $2$-cocycle $\omega : (\calF \times [0,1])^{(2)} \to \T$, 
which will give rise to a family of $2$-cocycles $\{\omega_t\}_{t\in[0,1]}$ on $\calF$ that is 
continuous in $t$.

Because $\Omega_\Lambda \rtimes \R^d$ satisfies the Baum--Connes conjecture with coefficients, 
\cite[Theorem 5.1]{Gillaspy15} applies, which says that the evaluation map 
$$
  q_t : C^*_r( (\Omega_\Lambda \rtimes \R^d)\times [0,1], \omega) \to C^*_r( \Omega_\Lambda \rtimes \R^d, \omega_t) 
$$
induces an isomorphism of $K$-theory groups. Composing this isomorphism with the Morita equivalence 
of $\Omega_\Lambda \rtimes \R^d$ with $\calG_\mathrm{Del}$, given a homotopy of $2$-cocycles 
$\sigma_\bullet$ on $\calF \times [0,1]$
(and so on $\calG_\mathrm{Del}\times [0,1]$ by restriction), we can consider finitely 
generated and projective modules $P_\bullet$ over $C^*_r(\calG_\mathrm{Del} \times [0,1], \sigma_\bullet)$. 
Composing with the evaluation map, $P_0$ is a free $C^*_r(\calG_\mathrm{Del}, \sigma_0)$-module 
if and only if $P_1$ is a free $C^*_r(\calG_\mathrm{Del}, \sigma_1)$-module.

Considering the magnetic twists of Example \ref{ex:mag_twist}, we can easily construct homotopies of 
$2$-cocycles via a continuous map $B_\bullet : \Omega_\Lambda \times [0,1] \to \bigwedge^2 \R^d$ 
which restricts to a continuous path $\{B_t\}_{t\in[0,1]}$ of magnetic fields. 
\end{remark}

\subsection{The localisation dichotomy} \label{subsec:Loc_dichotomy}
Decay properties of the normalised tight frame in Theorem \ref{thm:Delone_frame_or_basis} 
come from the seminorms on the dense submodules 
$\calS_k(\calF_{\Omega_0},\sigma)$ over the pre-$C^*$-algebras 
$\calA_k\subset C^*_r(\Omega_{\Lambda} \rtimes \R^d, \sigma)$ 
and $\calB_k \subset C^*_r(\calG_\mathrm{Del}, \sigma)$ for $k = 0, \cdots,\infty$. 
Because each $\calA_k$ is a pre-$C^*$-algebra, given 
$p=p^*=p^2 \in M_n(C^*_r(\Omega_{\Lambda} \rtimes \R^d, \sigma))$ and $0<\epsilon < 1$, there is 
some $p_k = p_k^*=p_k^2 \in M_n(\calA_k)\subset M_n(C^*_r(\Omega_{\Lambda} \rtimes \R^d, \sigma))$ 
such that $\|p-p_k\|<\epsilon$ in $C^*$-norm as well as a unitary $u_{k}$ in  
$M_n(C^*_r(\Omega_{\Lambda} \rtimes \R^d, \sigma))$ such that $p_{k}=u^{*}_{k}pu_{k}$ 
(see \cite[Section 4]{Blackadar}).
Consequently the finitely generated and projective $C^{*}(\calG_{\mathrm{Del}},\sigma)$-modules 
$pL^2(\calF_{\Omega_0},\sigma)^{\oplus n}_{\calG_{\mathrm{Del}}}$ and 
$p_{k} L^2(\calF_{\Omega_0},\sigma)^{\oplus n}_{\calG_{\mathrm{Del}}}$ are isomorphic. 
The module $p_{k} L^2(\calF_{\Omega_0},\sigma)^{\oplus n}_{\calG_{\mathrm{Del}}}$ 
contains the dense submodule $p_{k}\calS_{k}(\calF_{\Omega_0},\sigma)^{\oplus n}_{\calB_k}$. 
Note that we can choose $p_{k}=p_{\infty}$ for all $k\geq 1$.

By comparing $C^*$-module frames for 
$\calS_\infty(\calF_{\Omega_0}, \sigma) \subset \calS_1(\calF_{\Omega_0}, \sigma) 
\subset L^2(\calF_{\Omega_0}, \sigma)$, we can prove a weak version of the 
localisation dichotomy considered in~\cite[Section 5 (arXiv version)]{HaldaneWannier}.

\begin{prop}[Weak localisation dichotomy] \label{prop:weak_loc_dichotomy}
Let $p = p^* = p^2 \in M_n(C^*_r(\Omega_\Lambda \rtimes \R^d, \sigma))$ and $p_{k}=p_{k}^{*}=p_{k}^{2}\in M_{n}(\calA_{k})$ be equivalent projections as above.
 The following statements are equivalent.
 \begin{enumerate}
   \item[(i)] There is a $C^*$-module isomorphism 
   $p L^2(\calF_{\Omega_0},\sigma)^{\oplus n}_{\calG_\mathrm{Del}} \cong C^*_r(\calG_\mathrm{Del},\sigma)^{\oplus m}$.
   \item[(ii)] There are elements $\{w_j\}_{j=1}^m  \subset p_1 \calS_1(\calF_{\Omega_0},\sigma)^{\oplus n}_{\calB_1}$ 
   such that for all $\calL \in \Omega_0$, the collection $\{w_1^{(\calL,y)},\ldots, w_m^{(\calL,y)}\}_{y\in \calL}$ is an 
   orthonormal basis of $\pi_\calL(p) L^2(\R^d, \C^n)$ and for all $y \in \calL$,
\begin{equation} \label{eq:loc_dichotomy_sum}
      \sum_{j=1}^m \int_{\R^d} (1+ |x-y|^2) \big| w_j^{(\calL,y)}(x)\big|^2 \, \mathrm{d}x < \infty.
\end{equation}
   \item[(iii)] There exists 
   $\{e_j\}_{j=1}^m \subset p_\infty \calS_\infty(\calF_{\Omega_0},\sigma)^{\oplus n}_{\calB_\infty}$ 
   where for all $\calL \in \Omega_0$,  $\{e_1^{(\calL,y)},\ldots, e_m^{(\calL,y)}\}_{y\in \calL}$ is an 
   orthonormal basis of $\pi_\calL(p) L^2(\R^d, \C^n)$ with faster than polynomial decay.
 \end{enumerate}
\end{prop}
\begin{proof}
All statements except for Equation \eqref{eq:loc_dichotomy_sum} immediately follow from 
Theorem \ref{thm:Delone_frame_or_basis}. To see Equation \eqref{eq:loc_dichotomy_sum}, 
we note that the frame elements are such that $\|w_j \cdot \chi_p^* \|_1 < \infty$ 
for $j\in \{1,\ldots, m\}$, $p \in \R^d$ and $\|\cdot \|_1$ the seminorm on $\calS_1(\calF_{\Omega_0},\sigma)$, 
$$
   \| \xi \|_1 = \| \xi \| + \sum_{l=1}^d \| \nabla_l \xi \|.
$$ 
Passing 
to the localisation, the functions $w_j^{(\calL,y)}$ and 
$[\nabla_l(w_j \cdot \chi_y^*)]_\calL (x) = (x_l - y_l) w_j^{(\calL,y)}(x)$ are in $\pi_\calL(p) L^2(\R^d,\C^n)$ 
for any $y \in \calL$, $j \in \{1,\ldots, m\}$ and $l\in \{1,\ldots, d\}$. We can combine these cases to obtain 
Equation \eqref{eq:loc_dichotomy_sum}.  
\end{proof}

Proposition \ref{prop:weak_loc_dichotomy} should be compared to the Localisation Dichotomy Conjecture 
in~\cite[Section 5 (arXiv version)]{HaldaneWannier}.  We have shown that an $s_\ast$-localised Wannier basis 
for $s_\ast=1$ is equivalent to a Wannier basis with faster than polynomial decay, which in turn is equivalent to 
a free finitely generated and projective module. To improve condition (iii) to an 
exponentially localised Wannier basis will require more analytic arguments that fall outside the 
framework of pre-$C^*$-algebras we have considered. In Section \ref{subsec:Chern_and_Wannier}, we show 
that conditions (i)-(iii) of Proposition \ref{prop:weak_loc_dichotomy} imply that the (even) 
noncommutative Chern numbers vanish.

Since the submission of this manuscript, the preprints~\cite{PanatiDeloneWannier, LuStubbs21} have 
appeared that further develop the Localisation Dichotomy Conjecture for generalised Wannier bases in 
dimension $2$. In particular, \cite[Theorem 1]{LuStubbs21} shows that for gapped spectral projections of a 
magnetic Schr\"{o}dinger operator on $L^2(\R^2)$ (with mild regularity assumptions), an exponentially localised Wannier 
basis is equivalent to an $s_\ast$-localised Wannier basis with $s_\ast>5/2$. Hence, for projections 
$p \in C^*_r(\Omega_\Lambda \rtimes \R^2, \sigma)$ that fall into the framework of~\cite{LuStubbs21}, 
we can improve condition (iii) of Proposition \ref{prop:weak_loc_dichotomy} to exponential decay. 
The magnetic Schr\"{o}dinger operator with Delone atomic potential we consider below 
(see Equation \eqref{eq:Hamil_Delone_pot}) satisfies the regularity assumptions of~\cite{LuStubbs21} 
when $d=2$.

\begin{remark}[Connections to the Balian--Low Theorem]
Theorem \ref{thm:Delone_frame_or_basis} and Proposition \ref{prop:weak_loc_dichotomy} 
plays a similar role to the Balian--Low Theorem in time-frequency analysis. Brieflly, the theorem states 
that if a Gabor system $\{e^{2\pi i mt} g(t-n) \}_{m,n\in \Z}$ forms an orthonormal basis of $L^2(\R)$, then 
either $g$ or the Fourier transform $\hat{g}$ is such that the sum in Equation \eqref{eq:loc_dichotomy_sum} 
with $m=1$ diverges. By the work of Luef~\cite{LuefBL}, the Balian--Low Theorem can also be interpreted 
using finitely generated and projective modules over $C^*(\Z^2)$ and the fact that $C(\T^2)$ has no 
non-trivial projections. See~\cite{LuefBL} for more details and a generalisation to the rotation algebra 
$A_\theta\simeq C^*(\Z^2, \theta)$.
\end{remark}

\subsection{Wannier bases for Hamiltonians on $L^2(\R^d, \C^n)$ with Delone potentials}

We model a particle in $\R^d$ subject to a uniform magnetic 
field perpendicular to the sample. 
We take a magnetic potential $A=(A_1,\ldots,A_d)$ such that 
$A_j\in L^2_\text{loc.}(\R^d)$ and differentiable with
$$  
\frac{\partial}{\partial x_j} A_k - \frac{\partial}{\partial x_k} A_j = \text{const.} 
$$
for all $j,k\in\{1,\ldots,d\}$. 
For simplicity, we consider constant magnetic field strength but more general 
fields are possible (cf. Example \ref{ex:mag_twist}). The magnetic Schr\"{o}dinger operator 
is given by
$$  
H_0 =  \sum_{j=1}^d \!\left(-i\frac{\partial}{\partial x_j} - A_j \right)^2,   
$$
We choose the 
symmetric gauge and define 
$A_j = -\frac{1}{2}\sum\limits_{k=1}^d \theta_{j,k}x_k$ for $j=1,\ldots,d$, 
where $\theta_{j,k}$ is antisymmetric and real. 
Our choice of gauge gives the magnetic translations
$\{U_a\}_{a\in \R^d}$, where for any $a \in \R^d$, 
$$
 [H_0, U_a]=0, \qquad  (U_a\psi)(x) = e^{-i\langle x, \theta a \rangle} \psi(x-a), \,\,\, \psi \in L^2(\R^d).
$$
Given a compact space $\Omega$ with 
action $T: \R^d \to \mathrm{Homeo}(\Omega)$, we can define  the  groupoid $2$-cocycle
$$
 \theta: ( \Omega\rtimes \R^d)^{(2)}\to \mathbb{T}, \qquad  
  \theta( (\omega,x), (T_{-x}\omega, y)) = e^{-i\langle x, \theta y\rangle }, 
$$
which is normalised, $\theta( (\omega, x), (T_{-x}\omega, -x) ) = 1$, and constant over the unit space.

We wish to relate spectral properties of aperiodic Schr\"{o}dinger operators to the Delone 
groupoid. We do this by considering atomic potentials on point sets, 
\begin{equation}  \label{eq:Hamil_Delone_pot}
   H_\Gamma = H_0 +  \sum_{p \in \Gamma} v(\cdot - p),
   \qquad H_0 =   \sum_{j=1}^d \!\Big(-i\frac{\partial}{\partial x_j} - A_j \Big)^2, 
\end{equation}
where $v$ an atomic potential function. Provided the potential 
$V_\Gamma = \sum_{p \in \Gamma} v(\cdot - p)$ is essentially bounded, real valued and 
measurable, $H_\Gamma$ is essentially self-adjoint on the dense core $C_c^\infty(\R^d)$.
 We assume $\Gamma$ is $r$-discrete and 
restrict our potentials to the  $K$-subharmonic 
functions on $\R^d$,
$$
   L^1_{K,r}(\R^d) = \Big\{ f \in L^1(\R^d)\,:\, |f(x)| \leq Kr^{-d} \int_{|x-y|<r/2} \! |f(y)| \,\mathrm{d}y 
  \,\, \text{ for  a.e. }  x \Big\}.
$$

\begin{thm}[\cite{Bel92}] \label{thm:Delone_H_affiliated}
Let $\Lambda$ be an $(r,R)$-Delone set and $v\in L^1_{K,r}(\R^d)$ be a uniformly 
continuous and $\R$-valued atomic 
potential. Then the family of operators $h = \{ H_{\calL}\}_{\calL \in \Omega_{\Lambda}}$ with 
$H_{\calL}$ as in Equation \eqref{eq:Hamil_Delone_pot} is affiliated to the groupoid $C^*$-algebra
$C^*_r(\Omega_{\Lambda} \rtimes \R^d, \theta)$.
\end{thm}

\begin{remark}
Theorem \ref{thm:Delone_H_affiliated} is proved by showing that $h = \{ H_{\calL}\}_{\calL \in \Omega_{\Lambda}}$ 
is affiliated to a crossed product groupoid constructed from the continuous hull of the potential 
$V_\Lambda = \sum_{p \in \Lambda} v(\cdot - p)$. In general the 
continuous hull of a Delone set is topologically semi-conjugate to the continuous hull of a 
Delone potential via a surjective map $\Omega_{\Lambda} \to \Omega_{V_\Lambda}$. 
If  $\mathrm{supp}(v)\subset B(0, r_v)$ for some $r_v \leq r$, then the map is injective, 
see~\cite[Section 2.7]{BHZ00}.
\end{remark}

Let us now fix an $(r,R)$-Delone set $\Lambda$ and Hamiltonian of the form 
Equation \eqref{eq:Hamil_Delone_pot} such that the family of Schr\"{o}dinger operators 
$h = \{ H_{\calL}\}_{\calL \in \Omega_{\Lambda}}$ is affiliated to 
the twisted groupoid $C^*$-algebra $C^*_r(\Omega_{\Lambda} \rtimes \R^d, \theta)$. 
With the choice of twist $\theta$, we see that for 
$\pi_\calL : C^*_r(\Omega_{\Lambda} \rtimes \R^d, \theta) \to \calB[L^2(\R^d)]$,
\begin{align*}
  &(\pi_\calL(f)\psi)(x) = \int_{\R^d}\! e^{-i\langle x, \theta u\rangle} f(\calL-x,u-x) \psi(u) \,\mathrm{d}u, 
  &&U_a \pi_\calL(f) U_a^* = \pi_{\calL-a}(f).
\end{align*}

The twisted groupoid algebra $C^*_r(\Omega_{\Lambda} \rtimes \R^d, \theta)$ is Morita 
equivalent to $C^*_r(\calG_\mathrm{Del})$ and for every $(\calL, y) \in \calG_\mathrm{Del}$, 
we have a twisted translation action 
$$
  \psi^{(\calL,y)}(x) = e^{-i \langle x, By \rangle} \psi(x-y) = (U_y \psi)(x).
$$
Hence, the discrete groupoid translations are just the magnetic translations restricted to 
$\calL \in \Omega_0$.

Recall the pre-$C^*$-algebra 
$\calA = \calA_\infty \subset C^*_r(\Omega_{\Lambda}\rtimes \R^d, \theta)$ 
that comes from the Fr\'{e}chet completion 
of $C_c(\Omega_{\Lambda} \times \R^d, \theta)$ in the seminorms defined 
from the derivations $\{\partial_j\}_{j=1}^d$ (Proposition \ref{prop: apply_BC}).

\begin{lemma}
Let $h$ be a self-adjoint element affiliated to $M_n(C^*_r(\Omega_{\Lambda}\rtimes \R^d, \theta))$. 
Suppose 
that $\Delta \subset \sigma(h)$ is a bounded spectral region separated from $\sigma(h)\setminus \Delta$ with 
positive distance. Then $p_\Delta(h) = \chi_\Delta(h) \in M_n(\calA)$.
\end{lemma}
\begin{proof}
Because $\Delta$ is an isolated spectral region, $p_\Delta(h)$ can be approximated arbitrarily well 
by $\varphi(h)$ with $\varphi \in C_c^\infty(\R)$ such that $\varphi(x)=1$ for $x\in \Delta$ and 
$\varphi(x) = 0$ for $x\in \sigma(h)\setminus \Delta$. Hence $\varphi(h) \in M_n(\calA)$.
\end{proof}

Hence, we can adapt Theorem \ref{thm:Delone_frame_or_basis} to the case of Schr\"{o}dinger 
operators on $L^2(\R^d, \C^n)$ with Delone atomic potentials.

\begin{thm} \label{thm:Delone_Wannier_concrete}
Let $\Lambda$ be a $(r,R)$-Delone set and let $H_{\Lambda}$ be a magnetic
Schr\"{o}dinger operator on $L^2(\R^d, \C^n)$ with Delone atomic potential 
as in Equation \eqref{eq:Hamil_Delone_pot} with 
$v\in L^1_{K,r}(\R^d)$ and uniformly continuous. 
Suppose that $\Delta$ is an isolated and bounded spectral region 
of $\sigma(H_{\calL})$ for all $\calL \in \Omega_{\Lambda}$. 
Then there are elements $w_1,\ldots, w_m \in p_\Delta \calS_\infty(\calF_{\Omega_0}, \theta)^{\oplus n}$ 
such that for all $\calL \in \Omega_0$ 
the magnetic translates 
$\{ U_y w^\calL_1, \ldots, U_y w^\calL_m\}_{y \in \calL}$
give a normalised tight frame 
of $p_\Delta(H_\calL) L^2(\R^d, \C^n)$ with faster than polynomial decay.

The frame $\{ U_y w^\calL_1, \ldots, U_y w^\calL_m\}_{y \in \calL}$ is an orthonormal basis of 
$p_\Delta(H_\calL) L^2(\R^d, \C^n)$ for all $\calL \in \Omega_0$ if and only if 
the finitely generated and projective $C^*$-module 
$p_\Delta L^2(\calF_{\Omega_0},\theta)^{\oplus n}_{\calG_\mathrm{Del}} 
\cong C^*_r(\calG_\mathrm{Del},\theta)^{\oplus m}$.
\end{thm}
\begin{proof}
By the spectral gap assumption, the family of spectral projections 
$\{p_\Delta( H_\calL)\}_{\calL\in \Omega_{\Lambda}}$ give an element 
$p_\Delta(h) \in M_n(\calA)$. As such, we can apply Theorem \ref{thm:Delone_frame_or_basis} 
which gives the faster than polynomially decaying tight frame or orthonormal basis of the localisation 
$\pi_\calL(p_\Delta) \frakh_\calL^{\oplus n} = p_\Delta(H_\calL) L^2(\R^d, \C^n)$. 
\end{proof}

\begin{remarks}
\begin{enumerate}
  \item[(i)] The existence of an isomorphism 
$p_\Delta L^2(\calF_{\Omega_0},\theta)^{\oplus m}_{\calG_\mathrm{Del}} 
\cong C^*_r(\calG_\mathrm{Del},\theta)^{\oplus m}$ is a $K$-theoretic statement 
and implies that $[p_\Delta] = m[1] \in K_0(C^*_r(\calG_\mathrm{Del},\theta))$. 
  \item[(ii)] If we take a deformation of the magnetic field $\{\theta_t\}_{t\in[0,1]}$ such that 
  the $\R$-valued $2$-cocycle $\omega_t( x,y) = \langle x, \theta_t y \rangle$ is continuous 
  in $t$, we obtain a homotopy of $2$-cocycles in the sense of Remark \ref{rk:mag_field_htpy}. 
  Therefore, Theorem \ref{thm:Delone_Wannier_concrete} is stable under deformations of 
  the magnetic field provided that the spectral gap remains open throughout the deformation.
\end{enumerate}
\end{remarks}

\subsection{Obstruction to localised Wannier basis by (noncommutative) Chern numbers} \label{subsec:Chern_and_Wannier}

Let us briefly recall the periodic setting.
If the atomic potential $V_\Lambda = \sum_{p\in \Lambda} v(\cdot - p)$ is 
such that $\Lambda$ is a periodic and co-compact group $G$, then $H_\Lambda$ is 
affiliated to the algebra
$$
  C^*_r( \Omega_\Lambda \rtimes \R^d, \theta) \cong C^*_r( (\R^d / G) \rtimes \R^d, \theta ) \cong 
  C^*_r(G, \theta) \otimes \mathbb{K}.
$$
In the case $G = \Z^d$, the  the non-triviality of finitely generated and 
projective $C^*_r(\Z^d, \theta)$-modules with $\theta$ rational can be examined by studying the Chern classes 
of the spectral subspaces of the Hamiltonian viewed as a complex vector bundle over the 
Brillouin torus, $\wh{\Z}^d$. In the aperiodic setting, we can 
use tools from noncommutative geometry to carry out an analogous argument. Indeed, 
noncommutative Chern numbers for Hamiltonians affiliated to $C(\Omega)\rtimes_\theta \R^d$ and 
$C^*_r(\calG_\mathrm{Del},\theta)$ have already been studied~\cite{BRCont, BP17, BMes}.

Throughout this section, we will regularly take advantage of the equivalence between the 
continuous hull $\Omega_{\Lambda}\rtimes \R^d$ and $\calG_{\mathrm{Del}}$, which 
gives an isomorphism 
$K_0(C^*_r(\Omega_\Lambda \rtimes \R^d,\theta)) \cong K_0(C^*_r(\calG_\mathrm{Del}, \theta))$.

We first recall the top degree noncommutative Chern numbers for aperiodic or disordered 
magnetic Schr\"{o}dinger Hamiltonians with a spectral gap. To do this, we recall the trace 
per unit volume on $L^2(\R^d)$. Let $\Lambda_j$ be an increasing sequence of sets 
that converge to $\R^d$ in an appropriate sense, e.g. $\Lambda_j = [-j,j]^d$. Then 
for any $a \in \calB(L^2(\R^d))$, 
$$
  \Tr_\mathrm{Vol}( a) =  \lim_{j\to\infty}  \frac{1}{\mathrm{Vol}(\Lambda_j)} 
  \Tr( \Pi_{\Lambda_j} a), \qquad \Pi_{\Lambda_j}: L^2(\R^d) \to L^2(\Lambda_j).
$$

\begin{prop}[\cite{BRCont}]
Fix a probability measure $\bP$ on $\Omega_{\Lambda}$ that is invariant and ergodic under 
the $\R^d$-action and let $S_d$ denote the permutation group of $\{1,\ldots,d\}$. 
If $d>0$ is even and $p = p^* =p^2 \in M_n(\calA)$, then for almost all $\calL \in \Omega_\Lambda$ 
the functional
\begin{equation} \label{eq:NC_Chern_even}
C_d(p) =  \frac{(-2\pi i)^{d/2}}{(d/2)!} \sum_{\tau \in S_d} 
  (-1)^\tau \, (\Tr_{\C^n}\otimes \Tr_{\mathrm{Vol}}) \bigg(\pi_\calL(p)
     \prod_{j=1}^d [X_{\tau(j)},\pi_\calL(p)] \bigg),
\end{equation}
is integer valued and almost surely constant in $\Omega_{\Lambda}$.
\end{prop}

The number $C_d$ almost surely defines a homomorphism 
$K_0(C^*_r(\Omega_{\Lambda}\rtimes \R^d, \theta)) \to \Z$, which we can also 
consider as a homomorphism $K_0(C^*_r(\calG_\mathrm{Del},\theta)) \to \Z$. 
In particular 
$C_d(p) = C_d(p')$ if $[p]=[p'] \in K_0(C^*_r(\calG_\mathrm{Del},\theta))$ and 
$C_d(p)=0$ if $[p]=m[1] \in K_0(C^*_r(\calG_\mathrm{Del},\theta))$. 

For systems 
with $d\geq 3$, we may also wish to consider lower-dimensional invariants. These 
invariants are not integer-valued in general, but can still be used to study the 
topology of gapped spectral projections. 
We fix a probability measure $\bP$ on $\Omega_{\Lambda}$ that is 
invariant and under the $\R^d$-action. Then we recall the noncommutative calculus for 
$\calA \subset C^*_r(\Omega_{\Lambda} \rtimes \R^d, \theta)$,
\begin{align*}
  \calT( f ) = \int_{\Omega_{\Lambda}} \! f(\calL, 0) \,\mathrm{d}\bP, \qquad
  (\partial_j f)(\calL,x) = x_j f(\calL,x), \quad f \in \calA, \,\, j\in\{1,\ldots,d\}.
\end{align*}

\begin{prop}[\cite{BRCont}]
Let $p=p^* =p^2 \in M_n(\calA)$ and $\bP$ a probability measure on $\Omega_{\Lambda}$ that is 
invariant and under the $\R^d$-action. Then for any $k \leq d$ an even integer, the functional  
\begin{equation} \label{eq:weak_Chern}
  C_k(p) = \frac{(-2\pi i)^{k/2}}{(k/2)!} \sum_{\tau\in S_k} 
  (-1)^\tau \, (\Tr_{\C^n}\otimes \calT) \bigg( p
     \prod_{j=1}^k \partial_{\tau(j)} p \bigg),
\end{equation}
defines a homomorphism 
$K_0(C^*_r(\Omega_\Lambda \rtimes \R^d, \theta))  \to \R$.
If $\bP$ is ergodic under the $\R^d$-action and 
$k=d$, then $C_k(p) = C_d(p) \in \Z$ from Equation \eqref{eq:NC_Chern_even} almost surely.
\end{prop}

We again note that if $[p]=m[1] \in K_0(C^*_r(\calG_\mathrm{Del},\theta) )$, 
then  $C_k(p) = 0$ for any $k\geq 2$.
Combining our results on the noncommuative Chern numbers with Theorem \ref{thm:Delone_frame_or_basis}
and the weak localisation dichotomy (Proposition \ref{prop:weak_loc_dichotomy}),
we have the following. 

\begin{cor} \label{cor:nonzero_Chern_bad_Wannier}
Let $p=p^* =p^2 \in M_n(\calA)$ and $\bP$ a probability measure on $\Omega_{\Lambda}$ that is 
invariant under the $\R^d$-action. If $C_k(p)$ from Equation \eqref{eq:weak_Chern} 
is non-zero for some $k \geq 2$, then for any $\calL \in \Omega_0$ there can not be Wannier basis 
 of $\pi_\calL(p) L^2(\R^d,\C^n)$ constructed from magnetic translations in $\calL$ of 
elements $\{w_j\}_{j=1}^m \subset p\calS_1(\calF_{\Omega_0}, \theta)_{\calB_1}^{\oplus n}$ such that 
 for any $y \in \calL$
$$
   \sum_{j=1}^m \int_{\R^d} (1+ |x-y|^2) \big|  (U_y w_j^\calL)(x) \big|^2 \, \mathrm{d}x < \infty.
$$
\end{cor}

\subsection{Deformation of the Delone atomic potential}

We would like to consider the stability  of our results on aperiodic 
Schr\"{o}dinger operators $H_\Lambda$ when the Delone set $\Lambda$ is deformed 
(e.g. from an aperiodic set to a periodic lattice). Deforming a Delone set $\Lambda$ 
will change the crossed product groupoid and the $K$-theory may change substantially 
as the Delone set changes. However, we will show the pairings in cyclic cohomology considered in 
Section \ref{subsec:Chern_and_Wannier} are unaffected by such deformations.

\begin{lemma} \label{lem:continuity_lattice_potential}
Let  $v \in C_c(\R^d)$ be a continuous atomic potential with compact support.
If $\{\Lambda_t\}_{t\in [0,1]}$ is continuous path of $(r,R)$-Delone sets in 
$\mathrm{Del}_{(r,R)}$, 
then the path of Schr\"{o}dinger operators 
$\{H_{\Lambda_t}\}_{t\in [0,1]}$ is norm-continuous in the resolvent topology.
\end{lemma}
\begin{proof}
It is straightforward to see that for any $v\in C_c(\R^d)$, $v \in L^1_{K,r}(\R^d)$ and 
 the function 
$v_\Lambda(x) = \sum_{p \in \Lambda} v(p-x)$ is real-valued, measurable and essentially 
bounded for any $(r,R)$-Delone set $\Lambda$. 
Because $\Dom(H_{\Lambda})$ is constant for any $\Lambda \in \mathrm{Del}_{(r,R)}$, 
we can use the resolvent identity to bound
\begin{align*}
  \| (z- H_{\Lambda_1})^{-1}  - (z - H_{\Lambda_2})^{-1} \| 
  &= \| (z- H_{\Lambda_1})^{-1} (H_{\Lambda_1} - H_{\Lambda_2} ) (z- H_{\Lambda_2})^{-1} \|  \\
  &\leq \mathrm{ess.} \,\sup | v_{\Lambda_1} - v_{\Lambda_2} | \, 
  \| (z- H_{\Lambda_1})^{-1} \| \, \| (z- H_{\Lambda_2})^{-1} \|
\end{align*}
for any $z\in \C\setminus \R$. The result will therefore follow if we can show that 
the essential supremum is controlled by the topology on $\mathrm{Del}_{(r,R)}$. 
Suppose that $\mathrm{supp}(v) \subset B(0;M)$ for some $M>0$. Recalling the 
topology of $\mathrm{Del}_{(r,R)}$ (Proposition \ref{prop:Del_set_topology}) 
with $d_H$ the Hausdorff metric, we take $x\in \R^d$ and suppose that 
$d_H(\Lambda_1-x \cap B(0;M), \Lambda_2-x \cap B(0;M)) \leq \delta < r/2$. Taking 
$\delta$ small enough, we can ensure that 
$|\Lambda_1-x \cap B(0;M+r/2)| = |\Lambda_2-x \cap B(0;M+r/2)|$ and, furthermore, 
 we can 
decompose the sets $\Lambda_1-x \cap B(0;M+r/2)$ and $\Lambda_2-x \cap B(0;M+r/2)$ as 
pairs $(p,q) \in \Lambda_M:= \Lambda_1 \cap B(x;M+r/2) \times \Lambda_2 \cap B(x;M+r/2)$ 
such that $\|p-q\| \leq \delta$. We can therefore estimate
\begin{align*}
   |(v_{\Lambda_1} - v_{\Lambda_2})(x)| &= \Big| \sum_{p \in \Lambda_1} v(p-x) - \sum_{q\in \Lambda_2} v(q-x) \Big| 
   = \Big| \!\!\! \sum_{p \in \Lambda_1\cap B(x;M) }  \!\!\!    v(p-x) - 
     \!\!\!  \sum_{ q \in \Lambda_2 \cap B(x;M) } \!\!\!   v(q-x) \Big| \\
       &= \Big| \sum_{\substack{  (p,q) \in \Lambda_M  \\ \|p-q\| \leq \delta} } v(p-x) - v(q-x) \Big| 
       \leq  \sum_{\substack{  (p,q) \in \Lambda_M  \\ \|p-q\| \leq \delta} } \big|  v(p-x) - v(q-x) \big|.
\end{align*}
By continuity of $v$, given any $\epsilon>0$ we can choose a small enough $\delta$ so
that $\big|  v(p-x) - v(q-x) \big| <  \tfrac{\epsilon}{|\Lambda_1 \cap B(x;M+r/2)|}$ for all
$\|p-q\| \leq \delta$. Hence $|(v_{\Lambda_1} - v_{\Lambda_2})(x)| < \epsilon$ and the 
essential supremum is also bounded by $\epsilon$.
\end{proof}

\begin{defn}
Let $\{\Lambda_t\}_{t\in [0,1]}$ be a continuous path in $\mathrm{Del}_{(r,R)}$. We say 
that $\{H_{\Lambda_t}\}_{t\in[0,1]}$ is a gapped path if there exists a bounded interval 
$\Delta \subset \R$ such that for all $t \in [0,1]$ and $\calL_t \in \Omega_{\Lambda_t}$,
$\Delta\cap \sigma(H_{\calL_t})$ is non-empty and $\Delta$ is separated from the 
rest of the spectrum of $H_{\calL_t}$ by a positive distance.
\end{defn}

The conditions to obtain a gapped path are quite strong, but if satisfied give a path of 
operators $\{h_t\}_{t\in [0,1]}$ such that $h_t$ is affiliated to $C^*_r(\Omega_{\Lambda_t} \rtimes \R^d, \theta)$ 
and  $p_t = \chi_{\Delta}(h_t) \in \calA_t \subset C^*_r(\Omega_{\Lambda_t} \rtimes \R^d, \theta)$, 
a dense pre-$C^*$-algebra.

\begin{prop} 
Let $\{\Lambda_t\}_{t\in [0,1]}$ be a continuous path in $\mathrm{Del}_{(r,R)}$ and 
fix an atomic potential $v \in  C_c(\R^d)$. Suppose that 
$\{H_{\Lambda_t}\}_{t\in[0,1]}$ is a gapped path with isolated spectral region $\Delta \subset \R$. 
Then for $p_t = \{ \chi_{\Delta}(H_{\calL_t})\}_{\calL_t \in \Omega_{\Lambda_t}} \in \calA_t$ and 
any even integer $k\leq d$, the function 
$$
  [0,1] \ni t \mapsto C_k( p_t ) 
  = \frac{(-2\pi i)^{k/2}}{(k/2)!} \sum_{\tau\in S_k} 
  (-1)^\tau \, (\Tr_{\C^n}\otimes \calT) \Big(  p_t
     \prod_{j=1}^k \partial_{\tau(j)} p_t \Big) 
     \in \R
$$
is continuous, where $C_k(p)$ is the weak Chern number from Equation \eqref{eq:weak_Chern}.
\end{prop}
\begin{proof}
The assumption on the spectral gap implies that $C_k( p_t )$ is well-defined 
for all $t$. Because we have a uniform isolated spectral region $\Delta$, we can write for all 
$t\in [0,1]$,
$$
  p_t = \frac{1}{2\pi i} \oint_{\mathcal{C}} (z- h_t)^{-1} \,\mathrm{d}z, \qquad 
  h_t = \{(z- H_{\calL_t})^{-1}\}_{\calL_t \in \Omega_{\Lambda_t}}
$$
with $\calC$ a contour enclosing $\Delta$ and not intersecting any other part of the spectrum. 
If $\{\Lambda_t\}_{t\in [0,1]}$ is a continuous path in $\mathrm{Del}_{(r,R)}$, then 
there is a corresponding continuous path $\{\calL_t\}_{t\in [0,1]}$ with 
$\calL_\bullet$ in the orbit space of $\Lambda_\bullet$. 
Because $t\mapsto (z-H_{\calL_t})^{-1}$ is norm-continuous 
by Lemma \ref{lem:continuity_lattice_potential}, 
so is $t\mapsto \| (z- h_t)^{-1} \|$. By the integral formula for the 
spectral projections, $t\mapsto \|p_t\|$ is  continuous. 
Because the functional $C_k$ induces a weaker topology than the norm topology, 
$t\mapsto C_k(p_t)$ is also continuous.
\end{proof}

Continuity of the function 
  $[0,1]\ni t \mapsto \|(z-h_t)^{-1}\| \in \R$ for all $z$ in the resolvent 
  set implies that the spectral edges of $\sigma(h_t)$ are continuous away from gap closing points,
  see~\cite{BeckusBel}.

Continuity of the Chern numbers under deformations of Delone sets means that if the range of the 
pairing is quantised, then it is constant under deformations.
\begin{cor} \label{cor:top_Chern_const_lattice_deform}
Let $\{\Lambda_t\}_{t\in [0,1]}$ be a continuous path in $\mathrm{Del}_{(r,R)}$ and 
fix an atomic potential $v \in C_c(\R^d)$. Suppose that 
$\{H_{\Lambda_t}\}_{t\in[0,1]}$ is a gapped path with isolated spectral region $\Delta \subset \R$.  
Let $\bP_0$ be an invariant and 
ergodic probability measure on $\Omega_{\Lambda_0}$. Then for almost all 
$\calL_0 \in \Omega_{\Lambda_0}$, 
\begin{equation*} \label{eq:Z_valued_lattice_deform_const}
  C_d( p_t )  = C_d( p_{\calL_0} )
  := \frac{(-2\pi i)^{d/2}}{(d/2)!} \sum_{\tau \in S_d} 
  (-1)^\tau \, (\Tr_{\C^n}\otimes \Tr_{\mathrm{Vol}}) \bigg( \chi_{\Delta}(H_{\calL_0}) 
     \prod_{j=1}^d [X_{\tau(j)},\chi_{\Delta}(H_{\calL_0})] \bigg)
\end{equation*}
is integer valued and constant for all $t\in [0,1]$.
\end{cor}

 Corollaries \ref{cor:top_Chern_const_lattice_deform} and  \ref{cor:nonzero_Chern_bad_Wannier} 
 then give us the following stability result on delocalised Wannier bases.

\begin{cor}[Stability of delocalised Wannier basis under atomic deformations]
Let $\Lambda_0$ be an $(r,R)$-Delone set and consider $H_{\Lambda_0}$ with 
atomic potential $v \in  C_c(\R^d)$. Fix an invariant and ergodic probability measure 
on $\Omega_{\Lambda_0}$ and  suppose that 
$C_d( p_{\calL_0} ) \neq 0$ for almost all $\calL_0 \in \Omega_{\Lambda_0}$.
Then for any gapped path  $\{H_{\Lambda_t}\}_{t\in[0,1]}$ with isolated 
spectral region $\Delta \subset \R$ and any $\calL_t$ in the transversal $\Omega_{0,t}$,  
there can not be a faster than polynomially decaying Wannier basis of 
$\chi_{\Delta}(H_{\calL_t}) L^2(\R^d,\C^n)$ built from magnetic translates in $\calL_t$.
\end{cor}

We again note that by Proposition \ref{prop:weak_loc_dichotomy} 
the non-existence of a faster than polynomially decaying Wannier basis  
also implies that the weaker localisation bound of Equation \eqref{eq:loc_dichotomy_sum} diverges, 
  $$
   \sum_{j=1}^m \int_{\R^d} (1+ |x|^2) \big| w_j^\calL(x)\big|^2 \, \mathrm{d}x = \infty.
$$

\subsection*{Acknowledgments}
The authors thank Franz Luef, Domenico Monaco and  Guo Chuan Thiang for 
valuable feedback on an earlier version of this manuscript. We also thank 
Giovanna Marcelli, Massimo Moscolari and Gianluca Panati for sharing the 
results of~\cite{HaldaneWannier, PanatiDeloneWannier} with us. CB is supported by 
a JSPS Grant-in-Aid for Early-Career Scientists (No. 19K14548) and thanks 
the Mathematical Institute, Universiteit Leiden, for hospitality during the conference 
\emph{Noncommutative Geometry, Analysis, and Topological Insulators} in February 2020, 
where this work took shape. 
Both authors thank the Casa Matematica Oaxaca for hospitality and support 
during the workshop \emph{Topological Phases of Interacting Quantum Systems} 
in June 2019.

\end{document}